\documentclass[10pt]{amsart}

\usepackage[margin=3cm]{geometry}

\usepackage{amsthm, amssymb, amsmath, hyperref, dsfont, enumitem, tcolorbox, esint, mathtools}
\usepackage{todonotes}  
\usepackage{orcidlink}

\usepackage{tikz-cd,graphicx}
\usepackage[all]{xy}
\graphicspath{ {./images/} }

\mathtoolsset{showonlyrefs,showmanualtags}

\newcommand{\Mdef}[2]{\newcommand{#1}{\relax \ifmmode #2 \else $#2$\fi}}
\newcommand{\oline}{\overline}

\newcommand{\Exterior}{\mathchoice{{\textstyle\bigwedge}}%
    {{\bigwedge}}%
    {{\textstyle\wedge}}%
    {{\scriptstyle\wedge}}}

\newtheorem{theorem}{Theorem}[section]
\newtheorem{lemma}[theorem]{Lemma}
\newtheorem{prop}[theorem]{Proposition}
\newtheorem{definition}[theorem]{Definition}
\newtheorem{corollary}[theorem]{Corollary}

\newtheorem{remark}[theorem]{Remark}

\numberwithin{equation}{section}
\usepackage{contour}
\contourlength{0.8pt}

\DeclareMathOperator{\im}{Im}

\DeclareMathOperator{\ad}{ad}

\DeclareMathOperator{\pr}{pr}
\DeclareMathOperator{\Symp}{Symp}

\Mdef{\F}{\mathbb{F}}
\Mdef{\N}{\mathbb{N}}
\Mdef{\R}{\mathbb{R}}
\Mdef{\Z}{\mathbb{Z}}
\Mdef{\Q}{\mathbb{Q}}
\Mdef{\A}{\mathbb{A}}
\Mdef{\C}{\mathbb{C}}
\Mdef{\One}{\mathds{1}}
\Mdef{\RP}{\mathbb{R}\text{P}}

\newcommand{\ve}{\varepsilon}

\newcommand{\pder}[2][]{\frac{\partial#1}{\partial#2}}
\newcommand{\der}[2][]{\frac{\text{d}#1}{\text{d}#2}}
\title[Symplectification, Normal Cartan Connection, and Cartan Prolongations] {Symplectification of Rank 2 Distributions, Normal Cartan Connections, and Cartan Prolongations}
\date{\today}

\thanks{ I.\ Zelenko is partly supported by NSF grant DMS 2105528 and Simons Foundation Collaboration Grant for Mathematicians 524213.}

\author{Nicklas Day}
\address{Nicklas Day\\
	Department of Mathematics\\
	Texas A\&M University\\
	College Station\\
	Texas \ 77843\\
	USA}
\email{ncday@tamu.edu}
\urladdr{\url{https://sites.google.com/tamu.edu/nicklasday/home}}

\author{Boris Doubrov 
	\orcidlink{0000-0002-1930-1888}
}
\address{Boris Doubrov\\Belarusian State University\\ 
Nezavisimosti Ave.~4\\Minsk 220030\\Belarus}
\email{boris.doubrov@bsu.by}

\author{Igor Zelenko 
	\orcidlink{0000-0001-7900-2567}
} 
\address{Igor Zelenko\\
         Department of Mathematics\\
         Texas A\&M University\\
         College Station\\
         Texas \ 77843\\
         USA}
\email{zelenkotamu@tamu.edu}
\urladdr{\url{http://www.math.tamu.edu/~zelenko}}

\begin{document}
\begin{abstract}

We study the Doubrov--Zelenko symplectification procedure for rank $2$ distributions with $5$-dimensional cube---originally motivated by optimal control theory---through the lens of Tanaka--Morimoto theory for normal Cartan connections. In this way, for ambient manifolds of dimension \( n \geq 5 \), we prove the existence of the normal Cartan connection associated with the symplectified distribution. Furthermore, we show that this symplectification can be interpreted  as the $(n-4)$th iterated Cartan prolongation at a generic point. This interpretation naturally leads to two questions for  an arbitrary rank $2$ distribution with $5$-dimensional cube:
(1) Is the \((n-4)\)th iterated Cartan prolongation the minimal iteration where the Tanaka symbols become unified at generic points? (2) Is the \((n-4)\)th iterated Cartan prolongation the minimal iteration admitting a normal Cartan connection via Tanaka--Morimoto theory? Our main results demonstrate that: (a) For $n > 5$, the answer to the second question is positive (in contrast to the classical $n = 5$ case from $G_2$-parabolic geometries); (b) For  $n \geq 5$, the answer to the first question is negative: unification occurs already at the \((n-5)\)th iterated Cartan prolongation.
\end{abstract}
\maketitle

\section{Introduction}
 In their 2009 paper \cite{doubrov2009local}, B. Doubrov and I. Zelenko provided a novel approach to the local equivalence problem for rank 2 distributions called the {\it symplectification procedure}. Given a smooth rank $2$ distribution on an $n$ dimensional manifold (henceforth called a $(2,n)$ distribution) $D$ with $n>5$ satisfying a generic condition called {\it maximality of class}\footnote{To be defined in section \ref{sectionSympD}}, the symplectification procedure yields a canonical absolute parallelism on a fiber bundle over the original manifold. Because $D$ can be recovered from this absolute parallelism, its structure functions form a complete system of local invariants for $D$. In a recent article \cite{day2025canonicalframes}, it was shown that any bracket generating rank 2 distribution with 5-dimensional cube\footnote{The \textit{cube} of $D$ is the third piece of its weak derived flag, which is defined in section \ref{sectionTanakaSymbol}} is of maximal class at a generic point, effectively removing this condition on $D$.
 
 The main advantage of this approach over the standard Tanaka theory \cite{Tanaka1970} is that it gives {\it a uniform construction of canonical absolute parallelism independent of the Tanaka symbol of the original distribution}; in contrast to the standard Tanaka approach, the symplectification procedure does not require one to classify all possible Tanaka symbols nor assume their constancy.

One unusual aspect of the symplectification procedure is that, rather than directly applying Tanaka theory to construct a canonical absolute parallelism for the original distribution $D$, the method first constructs a new distribution called $\Symp(D)$ on a certain fiber bundle over the original manifold. The canonical absolute parallelism is then built for this new distribution. However, this modification does not alter the equivalence problem, since $D$ can be uniquely recovered from $\Symp(D)$.

The canonical frame constructed in \cite{doubrov2009local} is not in general a Cartan connection. In this paper, for $(2,n)$ distribution $D$ with $5$-dimensional cube and $n>5$, we will
\begin{enumerate}
    \item Replace the final steps in the symplectification procedure with an application of Tanaka--Morimoto theory to obtain a canonical Cartan connection. \cite{morimoto1993geometric}   
    \item Interpret the germ of the distribution $\Symp(D)$ as the $(n-4)$th iterated Cartan prolongation of $D$ at generic points;
    \item Show that similarly to $\Symp(D)$ the $(n-5)$th iterated Cartan prolongation of $D$ at a generic point has fixed Tanaka symbol independent of $D$, but lower iterated Cartan prolongation may have non-isomorphic Tanaka symbols; 
    \item There is a symbol of dimension $n$ such that for each distribution $D$ with this symbol and each $i<n-4$, Tanaka-Morimoto theory cannot be used to assign a normal Cartan connection to the $i$th iterated Cartan prolongation of $D$. This symbol is the one corresponding to the maximally symmetric $(2,n)$ distribution with $5$-dimensional cube. In this sense, the $(n-4)$th Cartan prolongation is the earliest one to which Tanaka-Morimoto theory can be applied.
    
\end{enumerate}
These objectives are accomplished in sections \ref{sectionNormalization}, \ref{sectionCorrespondence}, \ref{sectionNonexistence}, and \ref{sectionNonexistence} respectively.

In fact, we also prove items (2), and (3) for $n=5$. Besides, the case $n=5$, in contrast to $n>5$, is a parabolic geometry (of $G_2$-type): the $(2,5)$ distributions of maximal class are exactly those with small growth vector $(2,3,5)$, which are considered in the Cartan five-variable paper \cite{FiveVariables}. These distributions correspond to the grading of $G_2$ with marked shorter root; their first Cartan prolongations correspond to the Borel grading of $G_2$ (i.e. with both roots marked). Thus, a normal Cartan connection can be assigned both for $D$ and $\Symp(D)$ according to \cite{Tanaka1979}. Consequently, in the case of $n=5$, item (1) holds while item (4) does not hold.

Given a distribution of {\it constant symbol} with {\it linear invariant normalization condition}\footnote{To be defined in sections \ref{sectionTanakaSymbol} and \ref{sectionNormalizationCondition}, respectively}, Tanaka--Morimoto theory canonically produces a regular, normal Cartan geometry. In the case of $(2,n)$ distributions with $5$ dimensional cube (which is a necessary condition for maximality of class \cite{doubrov2009local}), there are $3$ non-isomorphic Tanaka symbols if $n=6$, then  $8$ non-isomorphic symbols if $n=7$, and continuous parameters appear in the set of all symbols for $n\geq 8$ for specific dimensions of the graded components. Therefore, for $n\geq 8$, distributions with certain fixed small growth vector are generically of non-constant symbol, thus Tanaka--Morimoto theory cannot be applied for them. Besides, for every $n\geq 6$ there are symbols which do not admit linear invariant normalization conditions (see Theorem \ref{Nonexistence Theorem} below), so that Tanaka--Morimoto theory cannot be applied for distribution with these symbols at every point as well.
However, the symplectified distribution $\Symp(D)$ has two exceptional properties:
 \begin{enumerate}[label=(\roman*)]
     \item 
     $\Symp(D)$ has constant symbol independent of $D$ with universal prolongation isomorphic to $\mathfrak{gl}_2(\R)\rtimes \mathfrak{heis}_{2n-5}$ (with a fixed grading), where $\mathfrak{heis}_{2n-5}$ is the $(2n-5)$-dimensional Heisenberg algebra. This property (along with the semidirect product structure on $\mathfrak{gl}_2(\R)\rtimes \mathfrak{heis}_{2n-5}$) is established in section \ref{sectionSympSymbol}. Briefly, $\mathfrak{gl}_2(\mathbb R)$ acts on $\mathfrak{heis}_{2n-5}$ as follows:  $\mathfrak{sl}_2(\mathbb R)$ acts irreducibly on a hyperplane $\mathcal E$ of $\mathfrak{heis}_{2n-5}$ transversal to its center and the identity matrix of  $\mathfrak{gl}_2(\R)$ acts as the identity on $\mathcal E$.
     \item The universally prolonged Tanaka symbol of $\Symp(D)$ admits a linear invariant normalization condition. This property is established in section \ref{sectionNormalization} using a criterion given by T. Morimoto in \cite{morimoto1993geometric}
 \end{enumerate}
 Because $\Symp(D)$ enjoys these two properties, Tanaka--Morimoto theory can be applied to $\Symp(D)$ to obtain not only a canonical frame, but a canonical Cartan connection.
 
 Thus for any $(2,n)$ distribution $D$ with $5$-dimensional cube and $n>5$, the $(n-4)$th iterated Cartan prolongation of $D$, which is locally equivalent to $\Symp(D)$ at a generic point, enjoys properties (i) and (ii). This leads naturally to the question: Do lower Cartan prolongations enjoy these properties? This question is answered in the following proposition:
 
 \begin{prop} 
    Let $D$ be a $(2,n)$ distribution with $5$-dimensional cube. Then at a generic point, the $(n-5)$th iterated Cartan prolongation $\pr^{(n-5)}(D)$ has Tanaka symbol which is independent of $D$, but this symbol does not admit a linear invariant normalization condition.
 
    On the other hand, at a generic point the $(n-4)$th Cartan prolongation $\pr^{(n-4)}(D)$ has Tanaka symbol which is independent of $D$, and this symbol admits a linear invariant normalization condition.
 \end{prop}

 In section 2, we lay out the basic constructions utilized throughout the paper: Tanaka symbols, Tanaka prolongations, Cartan connections, invariant normalization conditions, and Cartan prolongations. 

 In section 3, we review the symplectification procedure for rank 2 distributions originally constructed by B. Doubrov and I. Zelenko in \cite{doubrov2009local}, which treats the local equivalence problem for those rank 2 distributions at points satisfying a condition called maximality of class. In \cite{day2025canonicalframes}, we showed that any bracket-generating rank 2 distribution with $5$-dimensional cube is of maximal class at a generic point. Given such a rank 2 distribution $D$, the procedure yields a new rank 2 distribution $\Symp(D)$ to which Tanaka--Morimoto theory can be applied. $\Symp(D)$ is constructed by considering an even-dimensional subbundle $\mathcal{M}$ of the projectivized cotangent bundle which is naturally endowed with even contact structure. The kernel of the even contact structure is a line distribution $\mathcal{C}$ tangent to the so-called ``abnormal extremals'' of any optimal control problem on the space of curves tangent to the distribution $D$ (\cite{Liu-Sussmann,Agrachev-Sarychev, Zelenko99}). Osculating the fibers along the line distribution $\mathcal{C}$, we obtain an ascending flag that is complete on a generic subset $\mathcal{R}_D\subseteq \mathcal{M}$. The distribution $\Symp(D)$ on $\mathcal{R}_D$ is then constructed by taking skew-orthogonal complements of this flag with respect to the even contact structure. To conclude the section, we compute the Tanaka symbol for $\Symp(D)$ and show that this symbol admits a linear invariant normalization condition using a criterion by T. Morimoto in \cite{morimoto1993geometric}.

 In section 4, we construct the local equivalence between $\Symp(D)$ and $\pr^{(n-4)}(D)$, the $(n-4)$th Cartan prolongation of $D$. First, we demonstrate that the abnormal extremal trajectory corresponding to a point in $\mathcal{R}_D$ is determined by its $(n-4)$th jet. The equivalence is then established by assigning to an abnormal extremal trajectory in $\mathcal{R}_D$ its $(n-4)$th Cartan prolongation. 
 
 The final section addresses the natural question posed above. We show that for any distribution $D$ with $5$-dimensional cube, the distribution $\pr^{(n-5)}(D)$ has constant symbol which is independent of $D$, but this symbol does not admit a linear invariant normalization condition. Further, we show that for distributions $D$ of constant symbol isomorphic to the symbol of the most symmetric distribution with $5$-dimensional cube, each iterated Cartan prolongation $\pr^kD$ with $k\leq n-5$ has constant symbol at a generic point, but this symbol does not admit a linear invariant normalization condition. 

 \section{Preliminaries}
 \subsection{Tanaka Symbols and Tanaka Prolongations}
 \label{sectionTanakaSymbol}
 A {\it distribution} $D$ on a smooth manifold $M$ is a smooth linear subbundle of the tangent bundle $TM$. The {\it rank of $D$ at $q\in M$} is the dimension of the space $D(q)$; we assume distributions have constant rank.

For each $q\in M$ and each $i>0$, define

 \[
     D^{-1}(q) = D(q),\quad\text{and}\quad D^{-i}(q) = D^{-i+1}(q)+\Big\{[V,W](q): W\in \Gamma(D),V\in \Gamma(D^{-i+1})\Big\} 
 \]
 We call the descending flag $\{D^{-i}(q)\}_{i=0}^\infty$ the {\it weak derived flag of $D$ and $q$}; by convention, $D^{0}$ is the rank-zero subbundle of $TM$. If for some $i\in \N$, $D^{-i}(q) = T_qM$ at each $q\in M$, then we say $D$ is {\it bracket generating}. In this case, we call the smallest $\mu$ such that $D^{-\mu}(q)=T_qM$ the {\it step} of the distribution at $q$, and we call
 \[
    \Big(\text{rank}(D^{-1}(q)),\text{rank}(D^{-2}(q)),\ldots,\text{rank}({D}^{-\mu}(q))\Big)
 \]
 the {\it small growth vector} of $D$ at $q$. If the small growth vector is independent of $q$, we say $D$ is {\it equiregular}, and we call $D$ a $\Big(\text{rank}(D^{-1}(q)),\text{rank}(D^{-2}(q)),\ldots,\text{rank}({D}^{-\mu}(q))\Big)$ distribution. Often we are only interested in the first few entries of the small growth vector, omitting the subsequent entries.
 At each point $q\in M$, the {\it Tanaka symbol of $D$ at $q$} is the graded vector space associated to the weak derived flag
    \[
        \mathfrak{m}(q)=\bigoplus_{i=1}^\mu \mathfrak{m}_{-i}(q) = \bigoplus_{i=1}^\mu D^{-i}(q)/D^{1-i}(q).
    \]
 where $D^{-\mu}(q) = T_qM$. It is well-known that $\mathfrak{m}(q)$ inherits from the Lie bracket of vector fields the structure of a {\it nilpotent negatively graded Lie algebra}. In addition, $\mathfrak{m}(q)$ is {\it fundamental}; that is, $\mathfrak{m}(q)$ is generated as a Lie algebra by $\mathfrak{m}_{-1}(q)$. For more details, see \cite{Zelenko2009-qt} and \cite{cap2017}. The distribution $D$ is {\it of constant symbol} $\mathfrak{m}$ if for each $q\in M$, the Tanaka symbol $\mathfrak{m}(q)$ is isomorphic as a graded Lie algebra to the fixed graded Lie algebra $\mathfrak{m}$.

\begin{definition}
\label{universal_prol_def}
 The {\it universal Tanaka prolongation} of $\mathfrak{m}$ is the largest graded Lie algebra $\mathfrak{g}(\mathfrak{m}) = \bigoplus_i\mathfrak{g}_i$ satisfying
 \begin{enumerate}
     \item For each $i<0$, $\mathfrak{g}_i = \mathfrak{m}_i$
     \item If $X\in \mathfrak{g}_i$ for some $i\geq 0$ satisfies $[X,Y]=0$ for all $Y\in \mathfrak{g}_{-1}$, then $X=0$.
 \end{enumerate}
If $\mathfrak{m}$ is a Tanaka symbol, we shall briefly refer to $\mathfrak{g}(\mathfrak{m})$ as the {\it prolonged Tanaka symbol} of $D$. 
\end{definition}
Constructing the prolonged symbol $\mathfrak{g}(\mathfrak{m})$ more explicitly, one can recursively build piece $\mathfrak{g}_i$ of weight $i\geq 0$
 \begin{equation}
 \label{Tanaka prolongation formula}
    \mathfrak{g}_i=\Big\{ \varphi \in \bigoplus_{j=1}^{\mu} 
    \text{Hom}(\mathfrak{g}_{-j},\mathfrak{g}_{i-j})\Big|\ \varphi([X,Y]) = \big[\varphi(X),Y\big] + \big[X,\varphi(Y)\big]\Big\}
 \end{equation}
 Throughout, we write
 
 \[
    \mathfrak{g}^0 = \bigoplus_{j\geq 0} \mathfrak{g}_i
 \]

 \subsection{Regular Cartan Connections}
 We now define Cartan connections, following \cite[$\S$ 2.4]{cap2017}\footnote{Traditionally, a Cartan Connection is defined for a pair $(G,H)$, where $G$ is a Lie and a subgroup and $H\subseteq G$ is a Lie subgroup. For our purposes it suffices to define it for a Lie algebra $\mathfrak g$ and a Lie group $H$ whose Lie algebra is a subalgebra of $\mathfrak g$.}.
 
 \begin{definition}
     Let $H$ be a Lie group with Lie algebra $\mathfrak{h}$, and let $\mathfrak{h}$ be a subalgebra of a Lie algebra $\mathfrak{g}$. A {\it Cartan geometry of type $(\mathfrak{g},H)$} is a principal $H$-bundle $\mathcal P\to M$ with a $\mathfrak{g}$-valued 1-form $\omega\in\Omega^1(\mathcal{P};\mathfrak{g})$ satisfying
    \begin{enumerate}[label=(\roman*)]
         \item $\omega_\lambda: T_\lambda\mathcal{P}\to \mathfrak{g}$ is a linear isomorphism for each $\lambda\in \mathcal{P}$
         \item $R_h^*\omega = \text{Ad}(h^{-1})\circ\omega$ for all $h\in H$, where $R_{(-)}(-):H\times \mathcal{P}\to \mathcal{P}$ is the principal action
         \item $\omega(\zeta_X) = X$ for all $X\in \mathfrak{h}$, where $\zeta_X$ is the fundamental vector field corresponding to $X$
    \end{enumerate}
    The form $\omega$ is called a {\it Cartan connection of type $(\mathfrak{g},H)$}.
 \end{definition}
 A complete local invariant for a Cartan geometry $\mathcal{P}\to M$ of type $(\mathfrak{g},H)$ is given by its {\it curvature function}, which is a function
 \[
    \kappa:\mathcal{P}\to \Exterior^2 (\mathfrak{g}/\mathfrak{h})^*\otimes\mathfrak{g}
 \]
 At a point $p\in \mathcal{P}$, the curvature function $\kappa$ is defined on $X_1,X_2\in \mathfrak{m}$ by
 \[
    \kappa(p)\big(X_1,X_2) = [X_1,X_2]-\omega(p)\Big([\omega^{-1}(X_1),\omega^{-1}(X_2)]_p\Big)
 \]

 By property (i), the form $\omega$ is an absolute parallelism. Since the structure functions of this absolute parallelism can be deduced from the curvature function $\kappa$, this function constitutes a complete local invariant for the Cartan geometry.

 Now suppose that $\mathfrak{g}$ is a graded Lie algebra with non-negative part $\mathfrak{h}=\mathfrak{g}^0$ and negative part $\mathfrak{m}$; then $\mathfrak{g} = \mathfrak{h}\oplus\mathfrak{m}$. We can then identify
 \[
    \Exterior^2(\mathfrak{g}/\mathfrak{h})^*\otimes \mathfrak{g} =\Exterior^2\mathfrak{m}^*\otimes \mathfrak{g}  = C^2(\mathfrak{m},\mathfrak{g})
 \]
 as the second cochain module for $\mathfrak{m}$ with coefficients in $\mathfrak{g}$, which is an $\mathfrak{m}$-module under the adjoint representation. In general, the degree-$k$ cochain module is given by
 \[
    C^k(\mathfrak{m},\mathfrak{g})= 
    \Big(\bigwedge\nolimits ^k\mathfrak{m}^*\Big)\otimes \mathfrak{g}
 \]
 and the coboundary map $\partial$ defined on $\phi\in C^k(\mathfrak{m},\mathfrak{g})$ by
 \begin{gather}
    \partial(\phi)(X_0,X_1,\ldots, X_k) = \sum_{i} (-1)^i \big[X_i,\phi(X_0, X_1,\ldots, \widehat X_i, \ldots, X_k)\big] 
    \\
    + \sum_{i<j}\phi\big([X_i,X_j],X_0, X_1,\ldots, \widehat X_i,\ldots, \widehat X_j,\ldots, X_k\big).
 \end{gather}

 The grading on $\mathfrak{g}$ induces a grading on this complex; in particular, for $i_1,\ldots, i_k<0$ and $i_{k+1} \in \mathbb{Z}$, the subspace
 \[
    \mathfrak{g}_{i_1}^*\wedge \cdots \wedge \mathfrak{g}_{i_k}^*\otimes \mathfrak{g}_{i_{k+1}}
 \]
 has weight $\big(i_{k+1} - \sum_{j=1}^k i_j\big)$. The coboundary operator preserves this grading. Write $C^k_w(\mathfrak{m},\mathfrak{g})$ for the cochains of degree $k$ and graded weight $w$; also write $C^k_+(\mathfrak{m},\mathfrak{g})$ for the degree-$k$ cochains of positive graded weight. The complex also inherits an action of $\mathfrak{g}^0$ which is induced by the adjoint action.
 
 We call $\omega$ a {\it regular} Cartan connection if its curvature $\kappa$ takes values in $C^2_+(\mathfrak{m},\mathfrak{g})$. 

 \subsection{Invariant Normalization Conditions}
 \label{sectionNormalizationCondition}
 Now let $\mathfrak{m}$ be a Tanaka symbol with universal Tanaka prolongation $\mathfrak{g}=\mathfrak{g}(\mathfrak{m})$, and let $\mathfrak{g}^0$ be its non-negative part. Let $G_0$ be the group of grading preserving automorphisms of $\mathfrak{m}$, 
 $G^0:=G_0 \exp(\mathfrak{g}_+)$, where $\mathfrak{g}_+$ is the positively graded part of $\mathfrak{g}$, and the exponential is that of linear maps, viewing $\mathfrak{g}_+$ as in \eqref{Tanaka prolongation formula}.

 Under favorable conditions, Tanaka--Morimoto theory gives a method for assigning a regular Cartan connection of type $(\mathfrak{g},G^0)$ to each distribution of a fixed symbol. In order to make this assignment in a canonical way, one must make a choice of normalization.

 \begin{definition}
     A {\it linear invariant normalization condition for $\mathfrak{g}$} (or for $\mathfrak{m}$) is a subspace $\mathcal{N}\subseteq C_+^2(\mathfrak{m},\mathfrak{g})$ that is invariant under $\mathfrak{g}^0$, can be decomposed into weighted subspaces $\displaystyle{\mathcal{N} = \bigoplus_i \mathcal{N}_i}$ and satisfies
     \[
        C^2_+(\mathfrak{m},\mathfrak{g}) = \mathcal{N}\oplus \im\big(\partial\big|_{C^1_+(\mathfrak{m},\mathfrak{g})}\big)
     \]
     We often omit ``linear,'' referring to $\mathcal{N}$ simply as an invariant normalization condition.
 \end{definition}
  Once an invariant normalization condition $\mathcal{N}$ for $\mathfrak{g}$ has been fixed, we call a Cartan connection $\omega$ of type $(\mathfrak{g},G^0)$ {\it normal} if its curvature function $\kappa$ takes values in $\mathcal{N}$. 

 Given a linear invariant normalization condition for a symbol $\mathfrak{m}$, Tanaka--Morimoto theory assigns canonically to each distribution of constant symbol $\mathfrak{m}$ a Cartan geometry from which the distribution can be recovered. Reinterpreting Theorem 4.12 in \cite{cap2017} or Theorem 3.10.1 of \cite{morimoto1993geometric} in our context, we have the following theorem
 \begin{theorem}
 \label{Cartan Connection}
     Let $\mathfrak{m}$ be the a negatively graded, fundamental, nilpotent Lie algebra, and let
     \[
        \mathfrak{g}(\mathfrak{m}) = \mathfrak{m}\oplus \mathfrak{g}^0
     \]
     be its universal prolongation. Let $G^0$ be the Lie group $G_0\exp(\mathfrak{g}_+)$. Given a linear invariant normalization condition $\mathcal{N}$ for $\mathfrak{g}(\mathfrak{m})$, Tanaka--Morimoto theory assigns to each distribution of constant symbol $\mathfrak{m}$ a regular normal Cartan geometry of type $(\mathfrak{g},G^0)$ from which the distribution can be recovered.
 \end{theorem}
 
 Therefore, in the case that a linear invariant normalization condition can be constructed, Tanaka--Morimoto theory reduces the local equivalence problem to the comparison of Cartan curvatures.
 
 \subsection{Cartan Prolongation}
 \label{sectionCP}
 Cartan prolongations were originally developed in Cartan's papers \cite{FiveVariables} and \cite{Cartan1914-sz}, then refined in the work \cite{Bryant_Hsu} of R. Bryant and L. Hsu. Given a distribution $D$, an unparameterized curve is called {\it horizontal for $D$} if its tangent lines lie in $D$. Appending to such a curve its tangent lines yields its Cartan prolongation, which is an unparameterized curve in $\mathbb{P}D$. The Cartan prolongation of $D$ is the smallest distribution on $\mathbb{P}D$ containing the tangents of prolonged curves. 

\begin{definition}
    Let $D$ be a smooth rank 2 distribution, and let $M_1=\mathbb{P}D$ be the fiberwise projectivization of $D$. The {\it Cartan prolongation of $D$}, denoted $\pr D$, is the rank 2 distribution on $M_1$ such that for any $q\in M$ and any $\ell\in\mathbb{P}D(q)$,
     \[
         \big(\pr D\big)(q,\ell) = \big\{v\in T_{(p,\ell)}\mathbb{P}D: \pi_*(v) \in \ell\big\}
     \]
     where $\pi:\mathbb{P}D\to M$ is the natural projection.\footnote{In some sources, the Cartan prolongation is also called the {\it rank 1 prolongation}; e.g., \cite{SHIBUYA2009793}}
\end{definition}
     Further, for any smooth parameterized curve $\gamma: (-\ve,\ve)\to M$ horizontal to $D$ with nonvanishing derivative, the {\it Cartan prolongation of $\gamma$} is the curve in $M_1$ defined by $t\mapsto (\gamma(t), [\gamma'(t)])\in \mathbb{P}D$ and is denoted by $\pr\gamma$ or $\pr^1\gamma$.\footnote{In a similar way, one can also define Cartan prolongations of smooth unparametrized curves.}

     For $i>1$, we define recursively $\pr^iD = \pr(\pr^{i-1}D),$ with its canonical projection $\pi_{i}$ to $M_{i}= \mathbb{P}(\pr^{i-1}D)$. Also define recursively $\pr^i\gamma(t) = \pr(\pr^{i-1}\gamma)(t).$
     
  The Cartan prolongation $\pr(\gamma)$ is the unique lift of $\gamma$ to $M_1$ which is horizontal with respect to $\pr(D)$. 
  Let us make a few observations about Cartan prolongations.
 \begin{remark}
 \label{prolongation remark}
    \begin{enumerate}[label=(\roman*)]
        \item If $\gamma(t)$ has nonvanishing derivative, then $\pr\gamma(t)$ has nonvanishing derivative, so all higher prolongations are also defined. 
        \item The curve $\pr(\gamma(t))$ is {\it time-homogeneous}; that is, $\pr(\gamma(t))(\tau) = \pr(\gamma(t+\tau))(0)$ for any $\tau \in \R$.
        \item If $D$ is bracket generating, so too is $\pr D$.
        \item $\pr(D)^{-2}$ and $\pr(D)^{-3}$ are the pullbacks of $D$ and $D^{-2}$ under the projection $M_1\to M$, respectively. Consequently, if $D$ is bracket generating on a manifold of dimension greater than 2, $\pr(D)$ is a distribution with small growth vector $(2,3,4,\ldots)$.
        \item We can recover $(M,D)$ from the prolongation $(\pr M, \pr D)$; the fibers of the projection $\phi: \pr M\to M$ are the Cauchy characteristics of $(\pr D)^{-2}$. The distribution $D$ at a point $q\in M$ can then be recovered as
        \[ D(q) = \left\langle T_\lambda \phi\left(\pr D(\lambda)\right)\ \middle|\ \lambda \in \phi^{-1}(q)\right\rangle = T_{\lambda_0}\left(\left(\pr D\right)^{-2}(\lambda_0)\right) \]
        for any $\lambda_0\in \phi^{-1}(q)$. This recovery process, called deprolongation, demonstrates that local information is preserved by both prolongation and deprolongation. Observe that deprolongation can be applied to any distribution with small growth vector $(2,3,4,\ldots)$.
        \item The distribution $\pr^k D$ may not be equiregular \cite{MZ2010}. However, we shall only consider the generic subset of $M_k$ on which the prolonged distribution is equiregular.
    \end{enumerate}
 \end{remark}

 Let us show how deprolongation reduces the local equivalence problem for rank 2 distributions to the local equivalence problem for $(2,3,5,\ldots)$ distributions, which is treated by the symplectification procedure. Given any bracket generating rank 2 distribution $D$, one can apply the above deprolongation procedure until one of two situations arises:

\begin{enumerate}
    \item Deprolongation leads to a distribution with small growth vector $(2,3,5,\ldots)$. One can then apply the symplectification procedure, which will be introduced in the next section. The symplectification procedure yields a new rank 2 distribution $\Symp(D)$, from which $D$ can be recovered. Applying Tanaka--Morimoto theory to $\Symp(D)$ then yields a canonical normal Cartan connection, which provides a solution to the local equivalence problem \cite{Zelenko2009-qt}\cite{cap2017}
    
    \item Deprolongation leads to a bracket generating $(2,3)$ distribution. All such germs are equivalent to the 3-dimensional contact germ, so the original distribution is locally equivalent to the Goursat germ \cite{MZ2010} at a generic point.
 \end{enumerate}
 
 \section{The Symplectification Procedure}
 We now summarize the symplectification procedure originally constructed in \cite{doubrov2009local}, which will yield a new $(2,2n-4)$ distribution we call $\Symp(D)$. The procedure utilizes the natural contact structure on the projectivized cotangent bundle $\mathbb{P}T^*M$ to construct a so-called ``even contact structure'' on a submanifold $\mathcal{M}$ of codimension 3.  The kernel of this even contact structure is a canonical line distribution $\mathcal{C}$.

 \subsection{The Characteristic Line Distribution}
 Let $D$ be a  bracket- generating rank 2 distribution on a smooth manifold $M$ of dimension $n\geq5$. Define the annihilator of $D^{-\ell}$
 \[
     \big(D^{-\ell}\big)^\perp = \big\{ (p,q)\in T^*M : p\cdot v = 0\ \forall\ v\in D^{-\ell}(q)\big\}.
 \]
 
 Consider the fiberwise projectivization $\mathbb{P}T^*M$ of the cotangent bundle.
 Since each $(D^{-\ell})^\perp$ is a linear subbundle of $T^*M$, we may define a codimension 3 submanifold
 \[
     \mathcal{M} = \mathbb{P}\Big((D^{-2})^\perp\setminus (D^{-3})^\perp\Big)\subseteq \mathbb{P}T^*M.
 \]
 Let $\mathfrak{s}$ be the tautological (Liouville) one-form on $T^*M$; explicitly, for coordinates $(q^i)$ on $M$ with conjugate variables $p_i$, $\mathfrak{s}=\sum p_i\text{d}q^i$. Recall that $d\mathfrak{s}$ is the canonical symplectic form on $T^*M$. The form $\mathfrak{s}$ passes to a conformal class $\oline{\mathfrak{s}}$ of 1-forms on $\mathbb{P}T^*M$ which defines a contact structure.
 
 Since $\text{rank}(D^{-2})=3$, the submanifold $\mathcal{M}$ has codimension 3 in contact manifold $\mathbb{P}T^*M$. Restricting the contact forms $\oline{\mathfrak{s}}$ to $\mathcal{M}$ gives a hyperplane distribution \begin{equation} \label{evencontact_H} H=\ker\left(\oline{\mathfrak{s}}|_{\mathcal{M}}\right) \end{equation}
 with a conformal class of skew-symmetric forms $\oline{\sigma}=d\oline{\mathfrak{s}}|_{H}$ well-defined on this hyperplane distribution. Since $H$ is a hyperplane distribution, it has rank $2n-5$, so the kernel of the form $\oline\sigma$ must have odd rank. In \cite{doubrov2009local}, the authors show that $\ker(\oline{\sigma})$ has the minimal rank of 1, so that $\mathcal{M}$ is equipped with a so-called {\it even contact structure}. We shall write $\mathcal{C}$ for the line distribution $\ker(\oline{\sigma})$, called the {\it characteristic line distribution} of $D$. The integral curves of $\mathcal{C}$ are called {\it regular abnormal extremals} of the distribution $D$, and their projections onto $M$ are {\it regular abnormal extremal trajectories} (\cite{Liu-Sussmann,Agrachev-Sarychev, Zelenko99}). For brevity, we omit `regular' when referring to such curves. From this viewpoint, $\mathcal{M}$ can be viewed as the space of pointed regular abnormal extremal trajectories.

 \subsection{Maximality of Class}
 \label{sectionSympD}
Let $\pi:\mathcal M\to M$ be the canonical projection.
 The lift of $D$ to $\mathcal{M}$ is denoted by:
 \begin{equation}
 \label{J0}
     \mathcal{J}(\lambda) = \big\{v\in T_\lambda \mathcal{M} : \pi_*(v)\in D(\pi(\lambda))\big\},
 \end{equation}
 which is a distribution of rank $n-2$. Osculating with the characteristic line distribution $\mathcal{C}$, we obtain from $\mathcal{J}$ a flag at each point of $\mathcal{M}$. Write $\mathcal{J}^{(0)} = \mathcal{J}$ and define recursively
 \begin{equation}
     \mathcal{J}^{(i)}(\lambda) = \mathcal{J}^{(i-1)}(\lambda) + [\mathcal{C}, \mathcal{J}^{(i-1)}](\lambda)\quad \text{for}\ i\geq 1
 \end{equation}
 In the paper \cite{doubrov2009local}, the authors show that for each $0\leq i$ and each $\lambda\in \mathcal{M}$, we have 
 \begin{equation}
 \label{jump}
 \text{dim}\big(\mathcal{J}^{(i+1)}(\lambda)\big) - \text{dim}\big(\mathcal{J}^{(i)}(\lambda)\big) \leq 1.
 \end{equation}
 so that 
 \begin{equation}
 \label{max_dim_i}
 \dim
 \mathcal{J}^{(i)}(\lambda) \leq n-2+i.
 \end{equation}
 For the remainder of the paper, we restrict our considerations to the situation where the flag $\{\mathcal{J}^{(i)}(\lambda)\}$ is as complete as possible, in the following sense: 
 Let $H$ be as in \eqref{evencontact_H}. Since $\mathcal{C}$ is a Cauchy characteristic of $H$, and since $\mathcal{J}$ is contained within this distribution, so too is each $\mathcal{J}^{(i)}$. Hence 
 \begin{equation}
 \label{max_dim}
 \dim\mathcal{J}^{(i)}(\lambda)\leq  \text{rank}(H)= 2n-5. 
 \end{equation}
 Define the integer-valued functions on $\mathcal{M}$ and $M$, respectively:
 \begin{gather}
     \nu(\lambda) = \min\{i\in \N: \mathcal{J}^{(i+1)}(\lambda) = \mathcal{J}^{(i)}(\lambda)\}
     \\
     m(q) = \max\{\nu(\lambda): \lambda\in \pi^{-1}(q)\}
 \end{gather}
 One can show that the set $\{\lambda\in \pi^{-1}(q): \nu(\lambda)=m(q)\}$ is nonempty and Zariski open in the fiber $\pi^{-1}(q)$. The value $m(q)$ is called the {\it class} of $D$ at $q$. 
 
 Note that \eqref{max_dim_i} and \eqref{max_dim} imply that $\nu(\lambda)\leq n-3$, and therefore $m(q)\leq n-3$. The equality  $\nu(\lambda)=n-3$ holds if and only if $\mathcal{J}^{(n-3)}(\lambda)=H(\lambda)$. In the case that $m(q) = n-3$, we say that $D$ is {\it of maximal class} at $q\in M$; we say $D$ is {\it of maximal class} if $D$ is of maximal class at each point $q$. If $D$ is of maximal class at $q$, then $D$ is of maximal class on a neighborhood of $q$. Proposition 3.4 of \cite{Zelenko_2006} demonstrates that germs of $(2,n)$ distributions of maximal class are generic. Recently in \cite{day2025canonicalframes}, we showed the following theorem, which is even stronger:
 
 \begin{theorem}
     Let $D$ be a bracket generating rank 2 distribution on an $n$-dimensional manifold $M$, $n>5$. If the cube of $D$ has dimension $5$, then $D$ is of maximal class at a generic point of $M$.
 \end{theorem}
 
 Now define
 \[
     \mathcal{R}_D=\{\lambda\in \mathcal{M}: \nu(\lambda)=n-3\}.
 \]
  In the case that $D$ has $5$-dimensional cube, $\mathcal{R}_D$ is open and dense in the space $\mathcal{M}$ of pointed regular abnormal extremal trajectories. We shall construct the {\it symplectification} $\Symp(D)$ of $D$ to be a rank 2 distribution on the manifold $\mathcal{R}_D$, which has dimension $2n-4$.

\subsection{The Construction of $\Symp(D)$}
 If $D$ has $5$-dimensional cube, then by \eqref{jump} for any $\lambda\in \mathcal{R}_D$, we have:
 \[
     \mathcal{J}(\lambda)\subseteq\mathcal{J}^{(1)}(\lambda)\subseteq \cdots\subseteq \mathcal{J}^{(n-3)}(\lambda) = H(\lambda)\subseteq T_\lambda\mathcal{R}_D
 \]
 has the property that $\text{rank}(\mathcal{J}^{(i+1)})= \text{rank}(\mathcal{J}^{(i)})+1$ for each $0\leq i \leq n-4$. We can complete this flag using $\oline{\sigma}$: for each $i\geq 1$ and each $\lambda\in \mathcal{R}_D$, define
 \[
     \mathcal{J}_{(i)}(\lambda) = \big\{v\in T_\lambda \mathcal{R}_D : \oline{\sigma}(v,w) = 0\ \forall\ w\in \mathcal{J}^{(i)}(\lambda)\big\},
 \]
 the skew complement of $\mathcal{J}^{(i)}(\lambda)$ with respect to $\oline\sigma$. It is easy to see that $\oline{\sigma}\left(\mathcal{J},\mathcal{J}\right) = 0$, so we obtain a complete flag:

 \begin{equation}
     \label{calJFlag}
     \mathcal{C}(\lambda)=\mathcal{J}_{(n-3)}(\lambda)\subseteq \cdots \subseteq \mathcal{J}_{(1)}(\lambda)\subseteq \mathcal{J}(\lambda)\subseteq \mathcal{J}^{(1)}(\lambda)\subseteq \cdots \subseteq \mathcal{J}^{(n-3)}(\lambda)=H(\lambda)\subseteq T_\lambda\mathcal{R}_D.
 \end{equation}
 For each $0 < i < (n-3)$, we have

 \[
     \dim\big(\mathcal{J}_{(i)}(\lambda)\big) = n-2-i\quad\text{and}\quad \dim\big(\mathcal{J}^{(i)}(\lambda)\big) = n-2+i.
 \]

 Define the {\it symplectification of $D$} to be the rank 2 distribution $\Symp\big(D\big)$ on $\mathcal{R}_D$ such that $\Symp\big(D\big)(\lambda) = \mathcal{J}_{(n-4)}(\lambda)$. 

\subsection{Involutivity Conditions on $\Symp(D)$}

The lemmas in this section are proven in \cite{doubrov2009local} and provide involutivity conditions on $\Symp(D)$ that determine its Tanaka symbol. In particular, involutive distributions on $\mathcal{R}_D$ are obtained by intersecting the flag \eqref{calJFlag} with the vertical distribution for $\pi:\mathcal{R}_D\to M$. 

At each $\lambda\in \mathcal{R}_D$, one can show that
 \[
    \mathcal{J}^{(1)}(\lambda) = \big\{v\in T_\lambda \mathcal{M} : \pi_*(v)\in D^{-2}(\lambda)\big\}
 \]
 which implies (with some computation) that $\mathcal{J}_{(1)}(\lambda) = \ker(T_\lambda\pi)\oplus \mathcal{C}(\lambda)$. Define $V_1(\lambda) = \ker(T_\lambda\pi)$, the vertical subspace over $\lambda$. For $i=0$ and for $2\leq i\leq n-4,$ define
 \begin{equation}
 \label{Splitting}
     V_i(\lambda) = \mathcal{J}_{(i)}(\lambda)\cap V_1(\lambda),
 \end{equation}
 the vertical component of the $(n-2-i)$-dimensional piece of the flag at $\lambda$. Observe that $V_0(\lambda)=V_1(\lambda)$. From \eqref{calJFlag}, one can also observe that for each $1\leq i\leq n-4$

 \[
    \mathcal{J}_{(i)}(\lambda) = V_i(\lambda)\oplus \mathcal{C}(\lambda),
 \]
 so in particular $\Symp(D) = \mathcal{C}\oplus V_{n-4}$. The $V_i$ also  satisfy involutivity conditions; in the following lemma, which is Lemma 2 in \cite{doubrov2009local}, recall that $V_0=V_1$.

\begin{prop}
\label{involutivity}
    For each $i\geq 0$, we have involutivity conditions
    \begin{gather}
    \label{invol1}
        [V_i,V_i] \subseteq V_i
        \\
    \label{invol2}
        [V_i,\mathcal{J}^{(i)}] \subseteq \mathcal{J}^{(i)}
    \end{gather}
\end{prop}

\subsection{The Tanaka Symbol of $\Symp(D)$}
\label{sectionSympSymbol}
In this section, we will demonstrate the following proposition, which determines the Tanaka symbol of $\Symp(D)$, then we will compute the Tanaka prolongation of this symbol.

\begin{prop}
    Let $D$ be a rank 2 distribution with $5$-dimensional cube on a manifold of dimension $n\geq 5$. The Tanaka symbol of $\Symp(D)$ at any point is isomorphic to the algebra $\langle X\rangle \rtimes \mathfrak{heis}_{2n-5}$, where the semidirect product structure is given by 
    \begin{equation}
        [X,\ve_i]=\ve_{i+1},\text{ for }1\leq i \leq 2n-7,\quad [X,\ve_{2n-6}]=0,\quad  [X,\eta] = 0.
    \end{equation}
\end{prop}
\begin{proof}
    The weak derived flag and Tanaka symbol of $\Symp(D)$ can now be determined. Let $\Symp(D)^{-j}$ denote the $j$th piece of this weak derived flag. The decomposition \eqref{Splitting} and Proposition \ref{involutivity} imply 
    
    \[
        \Symp(D)^{-j} =
        \begin{cases}
            \mathcal{J}_{(n-3-j)} & \text{for}\ 1\leq j \leq n-4
            \\
            \\
            \mathcal{J}^{(j-n+3)} & \text{for}\ n-3\leq j \leq 2n-6
        \end{cases}
    \]
    
    \noindent Let $\tilde \ve_1$ and $\widetilde X$ be nonvanishing sections of $V_{n-4}$ and $\mathcal{C}$, respectively. Together, these vector fields form a frame for $\Symp(D)$. Now define recursively
    \[
        \widetilde \ve_i = \pi_i\Big([\widetilde X, \widetilde\ve_{i-1}]\Big)\ \text{for all}\ 2\leq i \leq n-4
        \quad \text{and}\quad 
        \widetilde \ve_i = [\widetilde X, \widetilde\ve_{i-1}]\ \text{for all}\ n-3\leq i \leq 2n-6
    \]
    where $\pi_i: \mathcal{J}_{(i)}\to V_{i}$ is the projection parallel to $\mathcal{C}$. Since $\Symp(D)$ is bracket generating, Proposition \ref{involutivity} gives
    \[
        V_k = \Big\langle \{\tilde \ve_\alpha\}_{\alpha=1}^{n-3-k}\Big\rangle\quad\text{for each}\ 1\leq k \leq n-4
    \]
    and
    \[
        \Symp(D)^{-j} = \Big\langle \widetilde X,\big\{\tilde \ve_{\alpha}\big\}_{\alpha=1}^{j}\Big\rangle\quad\text{for each}\ 1\leq j \leq 2n-6
    \]
    Proposition \ref{involutivity} gives relations
    \begin{gather}
    \label{TanakaRels1}
        [\widetilde X,\tilde \ve_\alpha] = \tilde \ve_{\alpha+1} \not\equiv 0 \pmod{\ \Symp(D)^{-\alpha}}\ \text{for}\ 1\leq \alpha \leq 2n-6
        \\
    \label{TanakaRels2}
        [\tilde \ve_\alpha,\tilde \ve_\beta] \equiv 0\pmod{\Symp(D)^{-\alpha - \beta}}\  \text{for}\ \alpha+\beta\leq 2n-6
    \end{gather}
    
    The final piece $\mathcal{J}^{(n-3)}=H(\lambda)$ is the even contact distribution on $\mathcal{M}$. Therefore, an additional contact vector field must be added to $\Symp(D)^{6-2n} = \mathcal{J}^{(n-3)}$ in order to form a frame on $\mathcal{M}$. To this end, define $\tilde\eta=-[\tilde \ve_1,\tilde \ve_{n-6}]$.  Indeed $\tilde \eta$ is not contained in $\mathcal{J}^{(n-3)} = \ker(\oline{\mathfrak{s}})$, since 
    
    \[
        \oline{\mathfrak{s}}([\tilde \ve_1,\tilde \ve_{n-6}]) = 0 \Longleftrightarrow \oline{\sigma}(\tilde \ve_1,\tilde \ve_{n-6}) = 0.
    \]
    
    For if the right-hand side holds, then $\tilde \ve_1$ is in $(\mathcal{J}^{n-3})^\angle = \langle \widetilde X\rangle$, which contradicts our choice of $\tilde\ve_1$. Hence, $(\widetilde X,\tilde \ve_1,\ldots, \tilde \ve_{2n-6},\tilde \eta)$ forms an adapted frame for the distribution $\Symp(D)$ on $\mathcal{R}_D$.
    
    Since $\widetilde X$ preserves $\mathcal{J}^{(n-3)}$, we have for any $1\leq \alpha\leq n-4$ that
    \begin{gather}
        \ad(\widetilde X)\big([\tilde\ve_\alpha,\tilde\ve_{2n-6-\alpha}]\big) = [\tilde\ve_{\alpha+1},\tilde\ve_{2n-6-\alpha}]+[\tilde\ve_{\alpha},\tilde\ve_{2n-5-\alpha}]\equiv 0\pmod{\Symp(D)^{6-2n}}
    \end{gather}
    Inducting on $\alpha$, we have for each $1\leq \alpha \leq n-3$ that
    \begin{equation}
    \label{TanakaRels3}
    [\tilde\ve_{\alpha},\tilde \ve_{2n-5-\alpha}] \equiv (-1)^{\alpha}\tilde \eta \pmod{\Symp(D)^{6-2n}}.
    \end{equation}
    Using the appropriate moduli, we define a basis for the Tanaka symbol
    \begin{gather}
        \oline X = \widetilde X,
        \\
        \oline \ve_\alpha=\tilde \ve_\alpha\ \pmod{ \Symp(D)^{1-\alpha}}\ \text{for}\ 1\leq \alpha\leq 2n-6
        \\
        \oline \eta = \tilde\eta\pmod{\Symp(D)^{6-2n}}.
    \end{gather}
    We have shown in \eqref{TanakaRels1}, \eqref{TanakaRels2}, and \eqref{TanakaRels3} that the only nontrivial relations in the Tanaka symbol are
    \begin{gather}
        [\oline X,\oline\ve_\alpha] = \oline\ve_{\alpha+1}\ \text{for}\ 1\leq \alpha\leq 2n-7\ \text{and}
        \\
        [\oline\ve_\alpha,\oline\ve_{2n-5-\alpha}] = (-1)^\alpha\oline\eta\ \text{for}\ 1\leq \alpha \leq n-3
    \end{gather}
    These are the desired relations, so the proposition is proven.
\end{proof}

We now define a semidirect product structure for $\mathfrak{gl}_2(\R)\rtimes\mathfrak{heis}_{2n-5}$. Let $\{Y,H,E,X\}$ be the standard basis of $\mathfrak{gl}_2(\R)$; this basis has nontrivial brackets
\[
    [X,Y] = H,\quad [H,X] = 2X,\quad [H,Y]=-2Y.
\]
Let $\mathfrak{heis}_{2n-5}$ be the $(2n-5)$-dimensional Heisenberg algebra with basis $\{\ve_1,\ldots, \ve_{2n-6},\eta\}$ and nontrivial brackets
\[
    [\ve_i,\ve_{2n-5-i}]=(-1)^i\eta \quad \text{for}\ 1\leq i \leq 2n-6
\]
Let $\mathfrak{gl}_2(\R)\rtimes \mathfrak{heis}_{2n-5}$ be the semidirect product structure defined by the algebra homomorphism $\phi:\mathfrak{gl}_2(\R)\to \mathfrak{der}(\mathfrak{heis}_{2n-5})$ given by
\begin{align}
    &\phi(H):\begin{cases} 
        \ve_i\mapsto (2i+5-2n)\ve_i,\ 1\leq i \leq 2n-5
        \\
        \eta\mapsto 0
    \end{cases}
    &&\phi(E):\begin{cases}
        \ve_i\mapsto \ve_i,\ 1\leq i \leq 2n-5
        \\
        \eta\mapsto 2\eta
    \end{cases}
    \\ 
    &\phi(X):\begin{cases}
        \ve_i\mapsto \ve_{i+1},\ 1\leq i \leq 2n-7
        \\ 
        \ve_{2n-6}\mapsto 0
        \\
        \eta\mapsto 0
    \end{cases}
    &&\phi(Y):\begin{cases}
        \ve_1\mapsto 0
        \\
        \ve_i\mapsto (i-1)(2n-5-i)\ve_{i-1},\ 2\leq i \leq 2n-6
        \\
        \eta\mapsto 0
    \end{cases}
\end{align}
Note that if $\mathcal E=\langle\ve_1, \ldots, \ve_{2n-6}\rangle$ then the map $a\mapsto \phi(a) |_{\mathcal E}$, $a\in \mathfrak{sl}_2(\R)$ defines the irreducible $\mathfrak{sl}_2(\R)$-representation on $\mathcal E$ so that $\phi(a)|_{\mathcal{E}}$ is an element of the symplectic algebra $\mathfrak{sp}(\mathcal E)$ with respect to the symplectic form $\omega$ on $\mathcal E$ defined by $[x, y]=\omega(x,y)\eta, \forall x, y\in \mathcal E$.

\begin{prop}
\label{Tanaka Symbol}
    Let $D$ be a rank 2 distribution with $5$-dimensional cube on an $n$-dimensional manifold, $n\geq 5$. 
    \begin{enumerate}
        \item If $n=5$, then the universally prolonged Tanaka symbol of $\Symp(D)$ is isomorphic to the split real form of the exceptional Lie algebra $G_2$ with grading obtained by marking both simple roots.
        \item If $n\geq 6$, then the universally prolonged Tanaka symbol of $\Symp(D)$ is isomorphic to $\mathfrak{gl}_2(\R)\rtimes \mathfrak{heis}_{2n-5}$ with the semidirect product structure defined by $\phi$ and with grading 
        \begin{gather}
        \label{SympGrading}
            \text{wght}(Y)=1,\quad \text{wght}(H)=\text{wght}(E)=0,\quad\text{wght}(X)=-1,
            \\
            \text{wght}(\ve_i)=-i,\quad \text{wght}(\eta) = -2n+5
        \end{gather}
    \end{enumerate}
\end{prop}
\begin{proof}
    For the case $n=5$, observe that the Tanaka symbol $\langle X \rangle \rtimes \mathfrak{heis}_5$ is isomorphic to the negative graded part of the exceptional Lie algebra $G_2$ with the given grading. Theorem 5.3 in \cite{Yamaguchi_1993} then gives that the universal prolongation is isomorphic to $\mathfrak{g}_2$.

    While the case $n\geq 6$ amounts to a direct computation from the definition of the universal Tanaka prolongation, we omit it for the sake of brevity. Instead, we provide a proof of a more general statement in Lemma \ref{prolong_k_lemma} below based on \cite[Theorem 3]{doubrov2009local}.
\end{proof}
In the sequel, we call the graded Lie algebra \eqref{SympGrading} the {\it prolonged symplectic symbol}.

\begin{remark}
One may wonder if the distribution $\Symp(D)$ is an arbitrary distribution with Tanaka symbol isomorphic to the symplectic symbol. In fact, the only constraints on $\Symp(D)$ are those given in Propositions \ref{involutivity} and \ref{Tanaka Symbol}. The involutivity condition \eqref{invol2} with $i=0$ is equivalent to the statement that $\text{Char}\big(\Symp(D)^{3-n}\big)$ has rank $n-4$. Given a distribution $\Delta$ with Tanaka symbol equal to the symplectic symbol such that $\text{Char}(\Delta^{3-n})$ has rank $n-4$, one can quotient by the characteristic. In the quotient, $\Delta^{3-n}$ gives a $(2,n)$ distribution with $5$-dimensional cube whose symplectification is again $\Delta$.

Furthermore, in the case $n=6$, the involutivity conditions of Proposition \ref{involutivity} are satisfied by any nonholonomic $(2,8)$ distribution $\Delta$ with constant Tanaka symbol as in Proposition \ref{Tanaka Symbol}. To see this, observe that the Lie bracket gives a skew-symmetric 2-form

\[
    \big(\Delta^{-3}\wedge\Delta^{-3}\big)\to \big(\Delta^{-4}/\Delta^{-3}\big)
\]
whose kernel $K$ meets $\Delta$ nontrivially. Since $K$ is even-dimensional, it must have dimension 2. Choose 
\begin{gather}
    V_0=V_1=K,
    \quad 
    V_2=\Delta\cap K,
    \quad \text{and}
    \quad
    \mathcal{J}^\alpha = \Delta^{-\alpha-3}\ \text{for}\ 0\leq \alpha\leq 2.
\end{gather} 
Viewing $V_1=K=\text{Char}(\Delta^{-3})$, we can see that $V_1$ and $V_2$ are involutive. Since $\mathcal{J}^1=\Delta^{-4}$
\[
    [V_1, \mathcal{J}^1] = \Big[V_1,\Delta^{-3}+[\Delta,\Delta^{-3}]\Big] = [V_1,\Delta^{-3}] + \Big[\Delta,[V_1,\Delta^{-3}]\Big] \subseteq \Delta^{-4} = \mathcal{J}^1.
\]
Hence, the involutivity conditions of Proposition \ref{involutivity} are satisfied for $i=0,1,2$. However, for $n>6$, the involutivity conditions are not satisfied in general.
\end{remark}

\subsection{Existence of Invariant Normalization Condition on $\Symp(D)$}
\label{sectionNormalization}
 We now use the work of Morimoto in \cite{morimoto1993geometric} to demonstrate the existence of a linear invariant normalization condition for the symplectic symbol $\mathfrak{gl}_2(\R)\rtimes\mathfrak{heis}_{2n-5}$ defined in section \ref{sectionSympSymbol} with grading \eqref{SympGrading}.

 If the Tanaka symbol of a distribution is semisimple, one can use the Killing form to construct an inner product with invariance properties; the cochains that are coclosed with respect to this inner product then comprise an invariant normalization condition. Morimoto's criterion, stated below, offers a slight generalization of the semisimple case. 
 
 Proposition 3.10.1 and Theorem 3.10.1 from \cite{morimoto1993geometric} give the following:

 \begin{prop}[Morimoto, 1992]
 \label{Morimoto}
 Let $\mathfrak{m} = \oplus_{p < 0}\mathfrak{m}_p$ be a graded Lie algebra with finite dimensional Tanaka prolongation $\mathfrak{g}$. Let $\mathfrak{g}^0$ be the non-negatively graded part of its Tanaka prolongation, and let $\mathfrak{k}\subseteq \mathfrak{g}^0$ be a subalgebra.

 Assume there exists a positive definite symmetric bilinear form $(\ ,\ )$ on $\mathfrak{g}$ and a linear map $\tau: \mathfrak{k}\to \mathfrak{g}$ satisfying
 \begin{enumerate}
     \item $(\mathfrak{g}_p,\mathfrak{g}_q) = 0$ if $p\neq q$
     \item $\tau(\mathfrak{g}_p)\subseteq \mathfrak{g}_{-p}$ for all $p\geq 0$.
     \item $\big([A,x],y\big) = \big(x,[\tau(A),y]\big)$ for all $x,y\in \mathfrak{g},A\in \mathfrak{k}$
 \end{enumerate}
 Then there exists a linear $\mathfrak{k}$-invariant normalization condition for $\mathfrak{g}$. 
 \end{prop}

 Choosing $\mathfrak{k}$ to be the non-negative part of the prolonged Tanaka symbol, Morimoto's criterion is easily satisfied for $\mathfrak{g}=\mathfrak{gl}_2(\R)\rtimes \mathfrak{heis}_{2n-5}$ for both the grading \eqref{Unification Theorem} and the grading \eqref{SympGrading}. Let $(\ ,\ )$ be the inner product on $\mathfrak{g}$ such that $\{Y,H,E,X,\ve_1,\ldots,\ve_{2n-6},\eta\}$ is an orthogonal basis for $\mathfrak{g}$, and 
 \begin{gather}
    |Y|^2 = |X|^2 = |\eta|^2 = 1,\quad |H|^2=|E|^2 = 2,\quad \text{and}\quad |\ve_i|^2 = \frac{(i-1)!}{(2n-6-i)!}
 \end{gather}
 Let $\tau:\langle Y,H,E\rangle\to \mathfrak{g}$ be the linear map defined by
 \[
    \tau: \begin{cases}
    Y\mapsto X
    \\
    H\mapsto H
    \\
    E\mapsto E
    \end{cases}
 \]
 One can easily check that $\tau$ and $(\ ,\ )$ satisfy Morimoto's criterion, thus proving the following theorem.

 \begin{theorem}
 \label{Existence Theorem}
     For each $n>5$, the graded Lie algebra $\mathfrak{g} = \mathfrak{gl}_2(\R)\rtimes \mathfrak{heis}_{2n-5}$ with grading as in \eqref{Unification Theorem} has a linear invariant normalization condition.
 \end{theorem}

 Similarly, for the case $n=5$, the universally prolonged Tanaka symbol of $\Symp(D)$ is the exceptional Lie algebra $G_2$ with grading obtained by marking both simple roots. Several sources including \cite{Tanaka1979} and \cite{CapSlovak} demonstrate that this Tanaka symbol admits a linear invariant normalization condition.

 In light of Proposition \ref{Tanaka Symbol} and Theorem \ref{Cartan Connection}, we now have a method for assigning a Cartan geometry to each $(2,n)$ distribution with $5$-dimensional cube and $n\geq 5$. Specifically, we should use the symplectification $\Symp(D)$, which has constant symbol as in Proposition \ref{Tanaka Symbol}. With the normalization condition from the above theorem, we can apply Theorem \ref{Cartan Connection} to obtain a Cartan geometry from which $D$ can be recovered.

 \section{The Symplectification Procedure via Cartan Prolongations}\label{Symp Via Cartan}
 
 \subsection{Unparameterized Jet Spaces}
 The foliations provided by the flag of involutive distributions $V_i$ separate abnormal extremal trajectories according to their jets. In order to make these notions rigorous, we now construct the manifold of unparameterized jets; that is, the orbit space of the jet bundle under the action of reparameterization. For each $i>0$, consider the subspace
    \[
        U^i(M) = \Big\{j^i_0(\gamma): \gamma\in C^\infty (\R,M)\ \text{with}\ \gamma'(0)\neq 0\Big\} \subseteq J^i_0(\R,M),
    \]
    where $J^i_0(\R,M)$ is the space of jets of maps $\R\to M$ evaluated at $0\in \R$. Consider the group of $i$-jets of smooth reparameterizations fixing the origin evaluated at $0\in \R$
    \[
        G^i:=\{j^i_0(\varphi): \varphi\in C^\infty(\R,\R), \varphi(0)=0,\varphi'(0)\neq 0\}
    \]
    with group composition law $j^i_0(\varphi_1)\cdot j^i_0(\varphi_2) = j^i_0(\varphi_1\circ\varphi_2)$. The group $G^i$ has a natural right action on $U^i(M)$, defined by
    \[
        j^i_0(\varphi).j^i_0(\gamma) = j^i_0(\gamma\circ \varphi)\quad \text{for any}\ j^i_0(\varphi)\in G^i, j^i_0(\gamma)\in U^i(M)
    \]
    Using standard arguments, one can prove that this action is smooth, free, and proper, so that $U^i(M)/G^i$ is naturally a smooth manifold. 
    
\subsection{Separation of Regular Abnormal Extremal Trajectories by Jets}

    We can now define for each $0\leq i \leq n-4$ a map $\rho^i$ that takes a point $\lambda\in \mathcal{R}_D$ to the unparameterized jet of the corresponding abnormal extremal trajectory
    \[
        \rho^i: \mathcal{R}_D\to U^i(M)/G^i; \lambda\mapsto [j_0^i(\pi e^{tC}\lambda)]
    \]
    for any nonvanishing section $C$ of $\mathcal{C}$. The map $\rho^i$ is independent of the choice of $C$ because rescaling $C$ reparameterizes the abnormal extremal trajectory and hence corresponds to acting by an element of $G^i$. The next proposition demonstrates that the level set foliation given by the map $\rho^i$ is precisely the foliation generated by $V_{i+1}$.

    \begin{prop}
    \label{Separating Abnormal Extremals}
        Fix $\lambda \in \mathcal{R}_D$. There exists an open neighborhood $W$ of $\lambda$ such that for any $\lambda_1,\lambda_2\in W$ and $0\leq i\leq n-4$, $\lambda_1$ and $\lambda_2$ are in the same leaf of the foliation of $W$ generated by $V_{i+1}$ if and only if $\rho^i(\lambda_1)=\rho^i(\lambda_2)$. By convention, $V_{n-3}$ is the rank-zero distribution, so $\rho^{n-4}$ is a local diffeomorphism onto its image.
    \end{prop}
    In order to prove the proposition, we shall prove a linearized version as a lemma. Once it is established, the proposition will be easily proven with the inverse function theorem. The following lemma is an improvement on the Lemma 2.12 in \cite{ProjEq}. 
    \begin{lemma}
    \label{Jet Kernel}
        For each $0\leq i \leq n-4$, let $\rho^i$ be the jet map defined above. At any $\lambda\in \mathcal{R}_D$, the kernel of the differential $T_\lambda\rho^i$ satisfies
        \[
            \ker(T_\lambda\rho^i) = V_{i+1}(\lambda)
        \]
        where $V_{n-3}(\lambda)$ is the zero subspace by convention.
    \end{lemma}

    \begin{proof}
        Fix a nonvanishing section $C$ of $\mathcal{C}$, and fix coordinates on $\mathcal{R}_D$ and $M$ so that
        \begin{enumerate}[label=(\roman*)]
            \item The projection $\mathcal{R}_D\to M$ is the projection onto some coordinate subspace
            \item The line of forms $\oline{\sigma}$ on the even contact distribution $H$ on $\mathcal{R}_D$ has a local section $\sigma$ which is constant in the given coordinates.
        \end{enumerate}
        We can find such coordinates by choosing an appropriate subset of the coordinates from a canonical system on $T^*M$. Let us first show that $\ker(T_\lambda\rho^i) \subseteq V_{i+1}(\lambda)$ for each $0\leq 1 \leq n-4$ and each $\lambda\in \mathcal{R}_D$. Let $v\in \ker(T_\lambda\rho^i)$ be arbitrary, and let $v= \der{s}\lambda_s|_{s=0}$ for some smooth curve $\lambda_s$ in $\mathcal{R}_D$. Since $v\in \ker(T_\lambda \rho^i)$, there exists a reparameterization $\varphi_s(t)$ with $\varphi_0(t)=t$ which is jointly smooth in $s$ and $t$ such that
        \begin{equation}
            \Big(\der{s}j^i_0\big(\pi \circ e^{\varphi_s(t)C}\lambda_s\big)\big|_{s=0}\Big) = 0
        \end{equation}
        In the specified coordinate chart, this expression can be written
        \begin{gather}
        \label{jet 1}
            \frac{\partial^{\ell+1}}{\partial t^\ell \partial s} \Big(\pi \circ e^{\varphi_s(t)C}\lambda_s\Big)\Big|_{s=0,t=0} 
            =
            T\pi\Big(\frac{\text{d}^\ell}{\text{d}t^\ell}\Big(\pder{s} e^{\varphi_s(t)C}\lambda_s\big|_{s=0}\Big)\Big|_{t=0}\Big)
            \\
            \label{jet 2}
            =
             T\pi\Big(\frac{\text{d}^\ell}{\text{d}t^\ell}\Big(e^{tC}_*(v) + \pder[\varphi_s(t)]{s}\big|_{s=0}\cdot C(e^{tC}\lambda)\Big)\Big|_{t=0}\Big) \quad \text{for all}\ 0\leq \ell \leq i.
        \end{gather}
        Since this expression is zero, we have in the same coordinates that
        \begin{equation}
        \label{VerticalEquivalence}
            \frac{\text{d}^\ell}{\text{d}t^\ell}e^{tC}_*(v) 
            \equiv 
            -\frac{\text{d}^\ell}{\text{d}t^\ell}\pder[\varphi_s(t)]{s}\cdot C\big(e^{tC}(\lambda)\big)\Big|_{t=0}
        \end{equation}
        We now show that $v\in \mathcal{J}_{(i+1)} = \mathcal{C}\oplus V_{i+1}$. Recall that $\mathcal{J}_{(i+1)} = \Big(\mathcal{J}^{(i+1)}\Big)^\angle$. By definition, an arbitrary element $w\in \mathcal{J}^{(i+1)}(\lambda)$ can be written as $w=\frac{\text{d}^j}{\text{d}t^j}e^{-tC}_*Y(e^{tC}\lambda)|_{t=0}$ for some vertical vector field $Y$ and some $0\leq j\leq i$.
        Because $C\in \ker(\sigma)$, the flow $e^{tC}$ preserves the form $\sigma$, and we can write
        \begin{gather}
        \label{sigma(v,w) 1}
            \sigma(v,w) 
            =
            \sigma\Big(v,\frac{\text{d}^j}{\text{d}t^j}e^{-tC}_*Y(e^{tC}\lambda)|_{t=0}\Big)
            =
            \frac{\text{d}^j}{\text{d}t^j} \sigma\Big(v,e^{-tC}_*Y(e^{tC}\lambda)\Big)\Big|_{t=0}
        \end{gather}
        Next, expand the derivative using the binomial theorem and use equation \eqref{VerticalEquivalence} to continue the equality
        \begin{gather}
        \label{sigma(v,w) 2}
            = \sum_{a=0}^j \binom{j}{a}\cdot \sigma\Big(\frac{\text{d}^a}{\text{d}t^a}e^{tC}_*(v)\big|_{t=0}, \frac{\text{d}^{j-a}}{\text{d}t^{j-a}}Y(e^{tC}\lambda)|_{t=0}\Big)
            \\
            = \sum_{a=0}^j \binom{j}{a}\cdot \sigma\Big(-\frac{\text{d}^a}{\text{d}t^a} \pder[\varphi_s(t)]{s}\cdot C(e^{tC}\lambda)\Big|_{t=0}, \frac{\text{d}^{j-a}}{\text{d}t^{j-a}}Y(e^{tC}\lambda)|_{t=0}\Big)
            \\
            =
             \frac{\text{d}^j}{\text{d}t^j} \sigma\Big(-\pder[\varphi_s(t)]{s}\cdot C(e^{tC}\lambda), Y(e^{tC}\lambda)\Big)\Big|_{t=0} = 0.
        \end{gather}
        Consequently, we have that $v\in \Big(\mathcal{J}^{(i+1)}\Big)^\angle = \mathcal{J}_{(i+1)}$. Since $v$ is vertical, this implies that $v\in V_{i+1}$, as desired. 

        Now let us show by induction on $i$ that $V_{i+1}(\lambda)\subseteq \ker(T_\lambda\rho^i)$. The case $i=0$ holds because $V_1$ is the vertical fiber for the projection $\pi:\mathcal{R}_D\to M$. Now assume for induction that $V_j(\lambda)= \ker(T_\lambda\rho^{j-1})$ for each $j\leq i$. Let $v\in V_{i+1}(\lambda)$ be arbitrary; we aim to show that $T_\lambda\rho^i(v) = 0$.
        
        As before, let $w\in \mathcal{J}^{(i+1)}(\lambda)$ be arbitrary, and write $w=\frac{\text{d}^j}{\text{d}t^j}e^{-tC}_*Y(e^{tC}\lambda)|_{t=0}$ for some vertical vector field $Y$ and some $0\leq j\leq i$.

        By assumption, $v\in \mathcal{J}_{(i+1)} = \Big(\mathcal{J}^{(i+1)}\Big)^\angle$, so $\sigma(v,w) = 0$. Following the equalities \eqref{sigma(v,w) 1} and \eqref{sigma(v,w) 2}, and then applying the induction hypothesis, we have that
        \begin{equation}
            0 
            =
            \sum_{a=0}^j \binom{j}{a}\cdot \sigma\Big(\frac{\text{d}^a}{\text{d}t^a}e^{tC}_*(v)\Big|_{t=0}, \frac{\text{d}^{j-a}}{\text{d}t^{j-a}}Y(e^{tC}\lambda)\Big|_{t=0}\Big)
            \\
            =
            \sigma\Big(\frac{\text{d}^j}{\text{d}t^j}e^{tC}_*(v)\Big|_{t=0}, Y(e^{tC}\lambda)\Big).
        \end{equation}
        Since $Y$ is an arbitrary vertical vector field and $0\leq j \leq i$ is arbitrary, this implies that $\frac{\text{d}^i}{\text{d}t^i}e^{tC}_*(v)\big|_{t=0}\in \mathcal{J}_{(1)}(\lambda) = \mathcal{C}(\lambda)\oplus V_1(\lambda)$. Let $\alpha_i\in \mathbb{R}$ be defined by
        \begin{equation}
        \label{alpha i}
            \frac{\text{d}^i}{\text{d}t^i}e^{tC}_*(v)\Big|_{t=0} = \alpha_i\cdot C(\lambda)\ \pmod{V_1(\lambda)}
        \end{equation}
        From the induction hypothesis, we know that $v\in V_{i}(\lambda) = \ker(T_\lambda\rho^{i-1})$, so we can assume by rescaling $C$ that for each $0\leq \ell \leq i-1$,
        \[
            \Big(\der{s} j^\ell_0\big(\pi \circ e^{tC}\lambda_s\big)\big|_{s=0}\Big)
            =
            \frac{\partial^{\ell+1}}{\partial t^\ell \partial s} \Big(\pi \circ e^{tC}\lambda_s\Big)\Big|_{s=0,t=0} = T_\lambda\pi\Big(\frac{\text{d}^{\ell}}{\text{d} t^\ell } e^{tC}_*(v)\Big|_{t=0}\Big)=0
        \]
        Now define a smooth reparameterization $\varphi_s(t) = t-\alpha_ist^i/i!$. In the previously specified coordinates, we can follow \eqref{jet 1} and \eqref{jet 2} to compute
        \begin{gather}
            \Big(\der{s} j^i_0\big(\pi \circ e^{tC}\lambda_s\big)\big|_{s=0}\Big)
            =
            \frac{\partial^{i+1}}{\partial t^i \partial s} \Big(\pi \circ e^{\varphi_s(t)C}\lambda_s\Big)\Big|_{s=0,t=0} 
            =
             T\pi\Big(\frac{\text{d}^i}{\text{d}t^i}\Big(e^{tC}_*(v) + \pder[\varphi_s(t)]{s}\big|_{s=0}\cdot C(e^{tC}\lambda)\Big)\Big|_{t=0}\Big)
        \end{gather}
        Substituting with \eqref{alpha i} and the definition of $\varphi_s(t)$, we continue the equality
        \begin{gather}
            = T\pi\Big(\alpha_i\cdot C(\lambda) - \frac{\text{d}^i}{\text{d}t^i}\frac{\alpha_it^{i}}{i!}\cdot C(e^{tC}\lambda)\Big|_{t=0}\Big)\Big) = T\pi\Big(\alpha_i\cdot C(\lambda) - \alpha_i\cdot C(\lambda)\Big) = 0.
        \end{gather}
        Therefore, $T_\lambda \rho^i(v) = 0$, as desired.
    \end{proof}

    Having proven the lemma, let us now prove Proposition \ref{Separating Abnormal Extremals}.

    \begin{proof}
        Applying the lemma in the case $i=n-4$ gives that for each $\lambda\in \mathcal{R}_D$, $T_\lambda\rho^{n-4}$ is injective. This implies $\rho^{n-4}$ is a local diffeomorphism onto its image. Now fix $0\leq i \leq n-4$, and let $W$ be an open neighborhood of $\lambda$ such that $\rho^{n-4}|_W:W\to J^{n-4}_0(\R,M)$ is a diffeomorphism onto its image.

        At each point $\lambda_1\in W$, we have that $V_{i+1}(\lambda_1) = \ker(T_{\lambda_1}\rho^i)$ is an involutive distribution. Possibly by shrinking $W$, let $W\xrightarrow{\xi} \widehat W$  be the quotient generated by this distribution. By the construction of $\widehat W$, $\rho^i$ passes to a well-defined map $\widehat\rho^i: \widehat W\to U^i(M)/G^i$. Further,
        \[
            \ker(T_{\xi(\lambda_1)}\widehat\rho^i) = T_{\lambda_1}\xi\Big(\ker (T_{\lambda_1}\rho^i)\Big) = 0.
        \]
        Therefore, $\hat \rho^i$ is a local diffeomorphism, and the result follows.
    \end{proof}

 \subsection{Constructing the Correspondence}
 \label{sectionCorrespondence}
 Recall that the space $\mathcal{R}_D$ is a generic subset of the space $\mathcal{M}$ of pointed regular abnormal extremal trajectories of $D$. Proposition \ref{Separating Abnormal Extremals} demonstrates that such curves can be distinguished by their $(n-4)$th unparameterized jet. Similarly, Cartan prolongation of horizontal paths distinguish these paths according to their first unparameterized jet: By construction, the prolongations of horizontal paths $\gamma_1$ and $\gamma_2$ agree at $t=0$ precisely if the first jets of $\gamma_1$ and $\gamma_2$ coincide up to linear reparameterization. That is,

 \begin{equation}
 \label{JC1}
     \Big([j^1_0(\gamma_1)] = [j^1_0(\gamma_2)]\in U^1(M)/G^1\Big) \Longleftrightarrow \Big(\pr(\gamma_1)(0) = \pr(\gamma_2)(0)\Big)
 \end{equation}

 The following proposition generalizes this statement to higher derivatives, showing that Cartan prolongations separate all the horizontal curves of a rank 2 distribution (including abnormal extremals) according to their jets.

 \begin{prop}
 \label{Separating Horizontal Curves}
     Let $D$ be a smooth rank 2 distribution on a manifold $M$, and let $\gamma_1, \gamma_2:(-\ve,\ve)\to M$ be smooth paths horizontal with respect to $D$ with nonvanishing derivatives. For $k\geq 0$, let $\pr^{k}\gamma_1$ and $\pr^k\gamma_2: (-\ve,\ve)\to M_{k}$ be the $k$th Cartan prolongations of $\gamma_1,\gamma_2$. Then 
     \begin{equation}
     \label{JC2}
         \Big([j^k_0(\gamma_1)] = [j^k_0(\gamma_2)]\in U^k(M)/G^k\Big)
         \Longleftrightarrow
         \Big(\pr^k(\gamma_1)(0) = \pr^k(\gamma_2)(0)\Big)
     \end{equation}
 \end{prop}

 \begin{proof} For $k=0$, the proposition holds trivially, so assume $k>0$. Let $\gamma_1$ and $\gamma_2$ be curves with $\gamma_1(0)=\gamma_2(0)$. Choose coordinates $(x^1,x^2,\ldots, x^{n})$ near $\gamma_1(0)$ so that $\gamma_1(t) = (t,0,\ldots, 0)$ for small $t$. Choose a frame $(X_1,X_2)$ for $D$ with $X_1=\partial_{x^1}$. Define a function $a^1$ on $\mathbb PD$ by
 \[
    a^1\big(\langle \alpha X_1+ \beta X_2\rangle\big) = \beta/\alpha
 \]
 Then $(x^1,x^2,\ldots, x^{n},a^1)$ form a local coordinate chart on $M_1=\mathbb{P} D$. The prolonged distribution $\pr(D)$ then has local frame
 \[
    X_{1,1} = a^1X_1+X_2, \quad X_{1,2} = \partial_{a^1}
 \]
 where $X_1$ and $X_2$ are implicitly identified with their lifts satisfying $X_1(a^1) = X_2(a^1) = 0$. Continuing this process, one can recursively define coordinates $(x^1,x^2,\ldots, x^{n}, a^1,\ldots, a^k)$ on $M_k = \mathbb{P}\big(\pr^{k-1}D\big)$ by
 \[
    a^i\big(\langle\alpha X_{i-1,1} + \beta X_{i-1,2}\big) = \beta/\alpha \quad \text{for all}\ 1< i\leq k
 \]
where $X_{i-1,1}$ and $X_{i-1,2}$ are the frame for $\pr^{i-1}(D)$ defined by
 \[
    X_{i-1,1} = a^{i-1}X_{{i-1},1}+X_{{i-1},2}, \quad X_{i-1,2} = \partial_{a^{i-1}}
 \]
 where $X_{i-1,1}$ and $X_{i-1,2}$ are implicitly identified with their lifts satisfying $X_{i-1,1}(a^{i-1})=X_{i-1,2}(a^{i-1})=0$. In these coordinates, it is easy to see that for small $t$, the $k$th prolongation of $\gamma_1(t)$ has coordinates
 \[
    \pr^k(\gamma_1)(t) = (t,0,0,\ldots, 0).
 \]
 Now parameterize $\gamma_2$ by $t=x^1$ so that $x^1\big(\gamma_2(t)\big) = t$, and write the tangent vector to $\gamma_2(t)$ as
 \[
    \gamma_2'(t) = X_1\big(\gamma_2(t)\big) + g(t)\cdot X_2\big(\gamma_2(t)\big)
 \]
 for some smooth function $g(t)$. We can compute inductively 
 \[
    a^k\Big(\pr^k(\gamma_2)(t)\Big) = g^{(k-1)}(t)
 \]
 Consequently, in this parameterization,
 \[
    \Big(j^k_0(\gamma_1)=j^k_0(\gamma_2)\Big)\Longleftrightarrow 
    \Big(g^{(i)}(0)=0\ \text{for all}\ 0\leq i\leq k-1\Big)\Longleftrightarrow
    \Big(\pr^k(\gamma_1)(0)=\pr^k(\gamma_2)(0)\Big)
 \]
Finally, we must show that
\[
    \Big(j^k_0(\gamma_1)=j^k_0(\gamma_2\circ\tau)\ \text{for a smooth reparameterization}\ \tau\ \text{with}\ \tau(0)=0\Big)\Longleftrightarrow \Big(j^k_0(\gamma_1)=j^k_0(\gamma_2)\Big)
 \]
 This can be seen by considering the coordinate projections $x^1\big(\gamma_1(t)\big)$ and $x^1\big(\gamma_2(t)\big)$; in order for the $k$th jet of these curves to agree at $t=0$, the reparameterization $\tau$ must vanish at zero to order $k$. Therefore, $j^k_0(\gamma_2\circ \tau) = j^k_0(\gamma_2)$.
\end{proof}

Having proven the correspondence of jets and Cartan prolongations, we are now prepared to prove the first main result.

\begin{prop}
\label{Local Equivalence}
    Fix a nonvanishing section $C$ of the line bundle $\mathcal{C}$. The map $\psi: \mathcal{R}_D\to M_{n-4}$ defined by $\lambda\mapsto \pr^{n-4}(\pi e^{tC}\lambda)(0)$
    is a local diffeomorphism at each point of $\mathcal{R}_D$, and for each $\lambda\in \mathcal{R}_D$,
    \begin{equation}
    \label{Distr Eq}
        T_\lambda\psi\big(\Symp(D)(\lambda)\big) = \pr^{n-4}D\big(\psi(\lambda)\big).
    \end{equation}
    Further, $T_\lambda\psi\big(V_{i}(\lambda)\big)$ is precisely the vertical distribution for the projection $M_{n-4}\to M_{i-1}$.\footnote{Although $\psi$ is a local diffeomorphism, it is not surjective; the singular points of the prolonged manifold $M_{n-4}$ are not in in image of $\psi$.}
\end{prop}

\begin{proof}

    Choose a nonzero $v\in V_{n-4}(\lambda)$. Then
    \[
        \Symp(D)(\lambda) = \langle C(\lambda), v\rangle
    \]
    In order to show \eqref{Distr Eq}, it suffices to show that $T_\lambda\psi\big(C(\lambda)\big)$ and $T_\lambda\psi\big(v\big)$ are linearly independent and contained in $\pr^{n-4}D\big(\psi(\lambda)\big)$. When this has been demonstrated, it will follow that $\psi$ is a local diffeomorphism near $\lambda$, since both $\Symp(D)$ and $\pr^{n-4}D$ are bracket generating.
    
    Recall from Remark \ref{prolongation remark} that Cartan prolongations are time-homogeneous; for any $\lambda \in \mathcal{R}_D$, we can compute
    \begin{equation}
    \label{Prolonged containment 1}
        T_\lambda\psi\big(C(\lambda)\big) 
        =
        \frac{d}{d\tau}\Big(\pr^{n-4}(\pi e^{tC}e^{\tau C}\lambda)(0)\Big)\Big|_{\tau=0}
        =
         \frac{d}{d\tau}\Big(\pr^{n-4}(\pi e^{tC}\lambda)(\tau)\Big)\Big|_{\tau=0}.
    \end{equation}
    where the prolongation in the first expression is with respect to the parameter $t$, considering $\tau$ to be constant. Since this is the derivative of a prolonged curve horizontal for $D$, it is tangent to $\pr^{n-4}D$ but is not contained in the vertical subspace for $M_{n-4}\to M_{n-5}$.
    
    Now let $\lambda_s$ be a curve in $\mathcal{R}_D$ tangent to $V_{n-4}$ with $v=\der{s}\lambda_s|_{s=0}$. Proposition \ref{Separating Abnormal Extremals} then gives that the curve of unparameterized jets 
    \[
        \rho^{n-5}(\pi e^{tC}\lambda_s) = [j^{n-5}_0(\pi e^{tC}\lambda_s)]\in U^{n-5}(M)/G^{n-5}
    \]
    is constant with respect to $s$. By Proposition \ref{Separating Horizontal Curves}, this implies that $\pr^{n-5}(\pi e^{tC} \lambda_s)(0)$ is constant with respect to $s$, so $T_\lambda\psi(v)$ is vertical for $M_{n-4}\to M_{n-5}$. Let us also show that $T_\lambda\psi(v)$ is nonzero. Again applying Proposition \ref{Separating Horizontal Curves}, the expression

    \begin{gather}
        T_\lambda\psi(v) 
        =
        \der{s}\Big(\pr^{n-4}\big(\pi e^{tC}\lambda_s\big)(0)\Big)\Big|_{s=0}
    \end{gather}
    is zero if and only if
    \begin{gather}
        \der{s}\Big(\big[j^{n-4}_0(\pi e^{tC}\lambda_s)\big]\Big)\Big|_{s=0} = T_\lambda\rho^{n-4}(v)
    \end{gather}
    is zero. By Lemma \eqref{Jet Kernel}, $T_\lambda\rho^{n-4}$ is an injection, so $T_\lambda\rho^{n-4}(v)$ is nonzero, and so too is $T_\lambda\psi(v)$. Since $T_\lambda\psi\big(C(\lambda)\big)$ is not vertical and $T_\lambda\psi(v)$ is nonzero and vertical, we have demonstrated \eqref{Distr Eq} and shown that $\psi$ is a local diffeomorphism.

    To prove the final statement, let $v = \der{s}\lambda_s|_{s=0}\in T_\lambda\mathcal{R}_D$ be arbitrary, and let $\pi^{i-1}:M_{n-4}\to M_{i-1}$ be the projection. The vector $v$ is vertical for $\pi^{i-1}$ precisely if the expression
    \begin{equation}
    \label{Tpsi}
        T\pi^{i-1}_{\psi(\lambda)}\circ T_\lambda\psi(v) = \der{s}\Big(\pr^{i-1}\big(\pi e^{tC}\lambda_s\big)(0)\Big)\Big|_{s=0}.
    \end{equation}
    is zero. Again by Proposition \ref{Separating Horizontal Curves}, \eqref{Tpsi} is zero if and only if
    \begin{equation}
    \label{Trho}
        \der{s}\Big(\big[j^{i-1}_0(\pi e^{tC}\lambda_s)\big]\Big)\Big|_{s=0} = T_\lambda\rho^{i-1}(v)
    \end{equation}
    is zero. By Proposition \ref{Separating Abnormal Extremals}, \eqref{Trho} is zero if and only if $v\in V_{i}(\lambda)$.
\end{proof}

The flag structure of $\Symp(D)$ also pushes forward over the local diffeomorphism $\psi$, determining the Tanaka symbol of $\pr^{(n-4)}D$ and the involutivity conditions for distinguished subdistributions.

\begin{corollary}
	The $(n-4)$th iterated Cartan prolongation $\pr^{n-4}D$ at a generic point has Tanaka symbol given by \eqref{SympGrading}. Further, there is a frame  $(\widetilde X,\widetilde \ve_1,\ldots \widetilde \ve_{2n-5},\widetilde \eta)$ on $M_{n-4}$ such that
    \[
        \big(\pr^{n-4}D\big)^i = \langle \widetilde X,\widetilde\ve_1,\ldots, \widetilde \ve_{i}\rangle\ \text{for}\ 1\leq i \leq 2n-5
    \]
    and which satisfies
	\begin{gather}
		\label{VV}
		\Big[\big\langle \widetilde \ve_1,\ldots, \widetilde \ve_{n-3-k}\big\rangle,  \big\langle \widetilde \ve_1,\ldots, \widetilde \ve_{n-3-k}\big\rangle\Big]\subseteq \big\langle \widetilde \ve_1,\ldots, \widetilde \ve_{n-3-k}\big\rangle \quad \text{and}
		\\
		\label{VJ}
		\Big[\big\langle \widetilde \ve_1,\ldots, \widetilde \ve_{n-3-k}\big\rangle,  \big\langle \widetilde X, \widetilde \ve_1,\ldots, \widetilde \ve_{n-3+k}\big\rangle\Big] \subseteq \big\langle \widetilde X, \widetilde \ve_1,\ldots, \widetilde \ve_{n-3+k}\big\rangle.
	\end{gather}
    for each $1\leq k \leq n-4$. This frame can also be constructed to satisfy
    \[
        \langle \widetilde \ve_1,\ldots, \widetilde\ve_i\rangle\ \text{is the vertical subspace for}\ M_{n-4}\to M_{n-4-i}\ \text{for}\ 1\leq i \leq n-4
    \]
\end{corollary}
This frame can be constructed by pushing forward the frame constructed in section 3 over the local diffeomorphism in Proposition \ref{Local Equivalence}. The conditions of the corollary are then precisely those of Proposition \eqref{involutivity}.

\section{Unification of Symbols and Existence of Invariant Normalization Conditions}

So far, we have interpreted the $\Symp(D)$ as the $(n-4)$th Cartan prolongation $\pr^{n-4}(D)$ and shown that this distribution enjoys the following advantageous properties:
\begin{enumerate}[label=(\roman*)]
    \item The Tanaka symbol of $\pr^{n-4}D$ is constant and independent of $D$ at a generic point. (In particular, its universal prolongation $\mathfrak{gl}_2(\R)\rtimes\mathfrak{heis}_{2n-7}$ with weights as in \eqref{SympGrading}.)
    \item The Tanaka symbol of $\pr^{n-4}D$ admits an invariant normalization condition.
\end{enumerate}
These properties allow us to use Theorem \ref{Cartan Connection} to assign a Cartan geometry to each $(2,n)$ distribution with $5$-dimensional cube. We now turn to the following question: Is there a $k<n-4$ such that $\pr^kD$ enjoys properties (i) and (ii)?

In Section \ref{sectionUnification}, we show that the Tanaka symbol of $\pr^k(D)$ is independent of $D$ for $k=n-5$, but not for any lower $k$. Finally, in section \ref{sectionNonexistence}, we demonstrate that for any $(2,n)$ distribution $D$ with the same Tanaka symbol as the most symmetric germ and for each $k<n-4$, the $k$th iterated Cartan prolongation $\pr^kD$ has Tanaka symbol not admitting a linear invariant normalization condition. Thus, even in the case of $\pr^{n-5}(D)$, where the symbol is independent of $D$, no linear invariant normalization condition exists.

Altogether, this shows that for an arbitrary $(2,n)$ distribution with $5$-dimensional cube, the iterated Cartan prolongation $\pr^{n-4}D$ (which is locally equivalent to $\Symp(D)$) is the earliest iterated Cartan prolongation to which a Cartan geometry can be assigned via Theorem \ref{Cartan Connection}. 

Let us now fix notation for a particular sequence of Tanaka symbols. Let $\mathfrak{s}^{k,n}$ be the negative part of the Lie algebra $\mathfrak{gl}_2(\R)\rtimes \mathfrak{heis}_{2n-5}$ with grading
\begin{gather}
\label{skn_symbol_grading}
    \text{wght}(Y)=1,\quad \text{wght}(H)=\text{wght}(E)=0, \quad \text{wght}(X) = -1,
    \\
    \text{wght}(\ve_i) = n-4-k-i,\quad \text{wght}(\eta) = -3-2k.
\end{gather}

In this language, Proposition \ref{Tanaka Symbol} can be rephrased as follows:\\
    {\it For any $(2,n)$ distribution $D$ with $5$-dimensional cube, the Tanaka symbol of $\Symp(D)$ is isomorphic to $\mathfrak{s}^{(n-4),n}$}

\subsection{The Unification of Tanaka Symbols}
\label{sectionUnification}

With the involutivity conditions \eqref{VV} and \eqref{VJ} in hand, we can demonstrate that the Tanaka symbol of $\pr^{n-5}D$ is independent of $D$.

\begin{theorem}
\label{Unification Theorem}
    Let $D\to M$ be a rank 2 distribution with $5$-dimensional cube. At a generic point of $M_{n-5}$, the prolonged Tanaka symbol of $\pr^{n-5}D$ is isomorphic to $\mathfrak{s}^{(n-5),n}$.
\end{theorem}

\begin{proof} Let $(\widetilde X,\widetilde \ve_1,\ldots, \widetilde \ve_{2n-6},\widetilde \eta)$ be a frame on $M_{n-4}$ adapted to $\pr^{n-4} D$ satisfying \eqref{VV} and \eqref{VJ}. Recovering $M_{n-5}$ and $\pr^{n-5}D$ as in section \ref{Symp Via Cartan}, choose a coordinate system $(x^1,\ldots, x^{2n-4})$ on $M_{n-4}$ so that $\widetilde\ve_1 = \pder{x^1}$. Since $\widetilde\ve_1$ is a characteristic vector field for $(\pr^{n-4}D)^2$, the projection $\phi:M_{n-4}\to M_{n-5}$ is simply the coordinate projection which omits the coordinate $x^1$. For any vector field $A$ on $M_{n-4}$, define $A' = T\phi(A|_{x^1=0})$, a vector field on $M_{n-5}$. Recovering $\pr^{n-5}D$, we have for all $\lambda\in \{x^1=0\}$ that 

\[
    \pr^{n-5}D\big(\phi(\lambda)\big) = T_{\lambda}\phi\Big(\big(\pr^{n-5}D\big)^{-2}(\lambda)\Big) = \Big\langle \widetilde X'\big(\phi(\lambda)\big),\widetilde\ve_2'\big(\phi(\lambda)\big)\Big\rangle
\]

\noindent For any vector fields $A=A^i\pder{x^i}$ and $B = B^j\pder{x^j}$ on $M_{n-4}$, one can easily check in coordinates that

\begin{equation}
    \label{(n-6) brackets}
     [A',B']_{\phi(\lambda )}=T_\lambda\phi\Big([A,B]_{\lambda} - A^1(\lambda)\cdot \big[\widetilde \ve_1,B\big]+B^1(\lambda)\cdot\big[\widetilde \ve_1,A\big]\Big).
\end{equation}
Applying the involutivity conditions \eqref{VV} and \eqref{VJ} and using \eqref{(n-6) brackets}, we get
\begin{gather}
    [\widetilde X',\widetilde\ve_i']_{\phi(\lambda)} 
    \equiv
    T_\lambda\phi\Big(\big[\widetilde X,\widetilde\ve_i\big]_\lambda\Big)
    \equiv
    \widetilde \ve_{i+1}'(\lambda)
    \ \pmod{\langle \widetilde X',\widetilde\ve_2',\ldots \widetilde\ve_i'\rangle_\lambda}
    \quad
    \text{for all}\ 2\leq i \leq 2n-7
    \\
    [\widetilde\ve_i',\widetilde\ve_j']_{\phi(\lambda)} 
    \equiv
    T_\lambda\phi\Big(\big[\widetilde \ve_i,\widetilde\ve_j\big]_\lambda\Big)
    \equiv
    0
    \ \pmod{\langle \widetilde X',\widetilde\ve_2',\ldots \widetilde\ve_j'\rangle_\lambda}
    \quad 
    \text{for all}\ 2\leq i<j\leq 2n-7, i+j<2n-5
    \\
    [\widetilde\ve_i',\widetilde\ve_{2n-5-i}']_{\phi(\lambda)} 
    \equiv
    T_\lambda\phi\Big(\big[\widetilde \ve_i,\widetilde\ve_{2n-5-i}\big]_\lambda\Big)
    \equiv
    \widetilde \eta'(\lambda)
    \ \pmod{\langle \widetilde X',\widetilde\ve_2',\ldots \widetilde\ve_{2n-5-i}'\rangle_\lambda}
    \quad 
    \text{for all}\ 2\leq i\leq n-3,
\end{gather}
These relations then imply
\begin{gather}
    \Big(\pr^{n-5}D\Big)^k({\phi(\lambda)}) = \Big\langle \widetilde X',\widetilde\ve_2,\ldots, \widetilde \ve_{k+1}\Big\rangle_{\phi(\lambda)}\quad \text{for all}\ 1\leq k\leq 2n-8
    \\
    \Big(\pr^{2n-5}D\Big)^{2n-7}({\phi(\lambda)})= \Big\langle \widetilde X',\widetilde \ve_2',\ldots, \widetilde \ve_{2n-6}', \widetilde\eta'\Big\rangle_{\phi(\lambda)} = T_{{\phi(\lambda)}}\big(M_{n-5}\big)
\end{gather}
and the Tanaka symbol of $\pr^{2n-5}D$ at $\phi(\lambda)$ is as desired.
\end{proof}
In fact, equation \eqref{(n-6) brackets} implies that the involutivity conditions \eqref{VV} and \eqref{VJ} hold for the frame $(X',\ve_2',\ldots, \ve_{n-6}',\eta')$, as long as one adds apostrophes and requires $i,j\geq 2$.

\begin{remark}
    Theorem \ref{Unification Theorem} does not hold if one replaces $(n-5)$ with $(n-6)$; that is, the Tanaka symbol for the $(n-6)$th Cartan prolongation of a distribution with $5$-dimensional cube depends on the distribution. In particular, one can consider the rank 2 distributions which arise from the Monge equations
    \[
        z' = (y^{(n-3)})^{2}\quad \text{and}\quad z'' = (y^{(n-4)})^2.
    \]
    After applying $(n-6)$ iterated Cartan prolongations to the pair of distributions, one obtains distinct Tanaka symbols.   
\end{remark}

\subsection{Nonexistence of an Invariant Normalization Condition on Lower Cartan Prolongations}
\label{sectionNonexistence}
We have shown in theorems \ref{Existence Theorem} and Proposition \ref{Local Equivalence} that for a $(2,n)$ distribution $D$ with $5$-dimensional cube, the $(n-4)$th iterated Cartan prolongation admits a linear invariant normalization condition at a generic point. We now show that the result is sharp; namely, for each $k<n-4$, we shall show that any distribution of constant symbol $\mathfrak{s}^{0,n}$ has no linear invariant normalization condition at a generic point of its $k$th Cartan prolongation.

Recall that given a fundamental symbol $\mathfrak m$ the {\it  flat distribution  $D(\mathfrak  m)$  of constant symbol (or type)  $\mathfrak m$} is defined to be the left-invariant distribution corresponding to the $-1$-graded component $\mathfrak m_{-1}$
on the simply connected Lie group with Lie algebra $\mathfrak{m}$.
\begin{remark}
\label{flat_symm}
As shown in \cite{Tanaka1970, Yamaguchi_1993}, if  the universal Tanaka prolongation $\mathfrak g(\mathfrak m)$ of $\mathfrak m$  is finite dimensional, then  the following two statements hold:
\begin{enumerate}
\item the germ of the flat distribution $D(\mathfrak m)$ has the infinitesimal symmetry algebra isomorphic to the universal Tanaka prolongation $\mathfrak{g}(\mathfrak m)$;
 \item The dimension of the infinitesimal symmetry algebra of any distribution with constant symbol is not greater than $\dim\, \mathfrak{g}(\mathfrak m)$,   and $D(\mathfrak m)$ is the unique, up to a local equivalence, distribution of constant symbol $m$  with the algebra of infinitesimal symmetries of dimension equal to  $\dim\, \mathfrak{g}(\mathfrak m)$.
\end{enumerate}
\end{remark}
\begin{lemma}
\label{prolongation_synbol_lemma}
    Let $E\to M$ be a distribution with constant Tanaka symbol $\mathfrak{s}^{k,n}$. At a generic point, the Cartan prolongation $\pr(E)$ has Tanaka symbol $\mathfrak{s}^{k+1,n}$. Furthermore, for any $0\leq k<n-4$, the Cartan prolongation of the Tanaka flat distribution $D(\mathfrak{s}^{k,n})$ at generic point is locally equivalent to the Tanaka flat distribution $D(\mathfrak{s}^{k+1,n})$. 
    \end{lemma}
\begin{proof}
Let
\begin{equation}
\label{skn frame}
    \big(\widetilde X,\widetilde \ve_{n-3-k},\widetilde\ve_{n-2-k},\ldots, \widetilde\ve_{2n-6},\widetilde\eta\big)
\end{equation}
be a frame on $M$ which satisfies $E = \langle \widetilde X, \widetilde \ve_{n-3-k}\rangle$ and
\begin{align*}
    &[\widetilde X, \widetilde \ve_i] 
     \equiv \ve_{i+1} 
    && \mod \langle \widetilde X, \widetilde \ve_{n-3-k},\ldots, \widetilde\ve_{i}\rangle 
    && \text{for}\ n-3-k\leq i \leq n-2+k
    \\
    &[\widetilde X, \widetilde \ve_i] 
     \equiv \ve_{i+1} 
    && \mod \langle \widetilde X, \widetilde \ve_{n-3-k},\ldots, \widetilde\ve_{i},\widetilde \eta\rangle 
    && \text{for}\ n-1+k\leq i \leq 2n-7
    \\
    &[\widetilde \ve_i,\widetilde\ve_j]
     \equiv 0
    && \mod \langle \widetilde X, \widetilde \ve_{n-3-k},\ldots,\widetilde \ve_{i+j-n+3+k}\rangle
    &&  \text{for}\ n-3-k\leq i < j\ \text{with}\ i+j<2n-5
    \\
    &[\widetilde \ve_i,\widetilde\ve_{2n-5-i}] 
     \equiv \eta 
    && \mod \langle \widetilde X, \widetilde \ve_{n-3-k},\ldots,\widetilde \ve_{n-2+k}\rangle 
    &&  \text{for}\ n-3-k\leq i \leq n-3
    \\
    &[\widetilde \ve_i,\widetilde\ve_j] 
     \equiv 0 
    && \mod \langle \widetilde X, \widetilde \ve_{n-3-k},\ldots,\widetilde \ve_{i+j-n+3+k},\eta\rangle
    &&  \text{for}\ n-3-k\leq i < j\leq 2n-6\ 
    \\
    & && && \quad \text{with}\ 2n-5 < i+j
    \\
    &[\widetilde X,\widetilde \eta] 
     \equiv 0 
    && \mod \langle \widetilde X,\widetilde\ve_{n-3-k},\ldots, \widetilde \ve_{n-1+k},\widetilde\eta\rangle
    \\
    &[\widetilde \ve_i,\widetilde \eta]
     \equiv 0 
    && \mod \langle \widetilde X,\widetilde\ve_{n-3-k},\ldots, \widetilde\ve_{i+2+2k},\widetilde\eta\rangle
    &&   \text{for}\ n-3-k\leq i \leq 2n-6
\end{align*}
For some smooth function $c$, we can then write
\begin{equation}
    [\widetilde X,\widetilde\ve_{n-1+k}] \equiv \widetilde \ve_{n+k} + c\cdot \widetilde \eta \mod \langle \widetilde X,\widetilde \ve_{n-3+k},\ldots, \widetilde \ve_{n-1+k}\rangle
\end{equation}
 Define a function $a$ on a generic subset of $\mathbb{P}E$ by

\[
    a\big(\langle \alpha \widetilde X + \beta\widetilde\ve_{n-3-k}\rangle\big) = \beta/\alpha.
\]
Define $\partial_a$ to be the vector field vertical with respect to the projection $\mathbb{P}E\to M$ with $da(\partial_{a})=1$. For each vector field $\widetilde A$ in the frame \eqref{skn frame} on $M$, identify $\widetilde A$ with the vector field on $\mathbb{P}D(\mathfrak{s}^{k,n})$ which lifts $\widetilde A$ and is annihilated by $da$. Then $(\widetilde\partial_a, \widetilde X, \widetilde\ve_{n-3-k}, \widetilde\ve_{n-2-k},\ldots, \widetilde\ve_{2n-6}, \widetilde\eta)$ forms a frame for the manifold $E$. Define
\begin{gather*}
    \hat X =  X + a \widetilde\ve_{n-3-k},\quad \hat \ve_{n-4-k} = -\partial_a
    \\
    \hat \ve_{i} =  \widetilde\ve_{i}\ \text{for}\ n-3-k\leq i \leq 2n-6\ \text{with}\ i\neq n-1+k,n+k
    \\
    \hat \ve_{n+k-1} =  \widetilde\ve_{n+k-1}\ + (-1)^{n+k+1}a\widetilde\eta,\quad \hat \eta =  \widetilde\eta
    \\
    \hat\ve_{n+k} = \widetilde\ve_{n+k} + c\cdot \widetilde\eta
\end{gather*}
where $c$ is the function defined by
\[
    \big[\hat X, \hat \ve_{n-1+k}\big] 
    \equiv 
    \widetilde \ve_{n+k} +  c\cdot\widetilde\eta 
    \mod \big\langle\widetilde X, \widetilde\ve_{n-3-k},\ldots,\widetilde\ve_{n-2+k},\widetilde\ve_{n-1+k} + (-1)^{n+k+1}a\widetilde\eta\big\rangle
\]
In this language, $\pr 
(E)= \langle \hat X, \hat \ve_{n-4-k}\rangle$, and the frame 
\begin{equation}
\label{sk+1n frame}
    \big( \hat X, \hat \ve_{n-4-k},\hat \ve_{n-3-k},\ldots, \hat \ve_{2n-6},\hat \eta\big)
\end{equation}
enjoys relations
\begin{align*}
    &[\hat X, \hat \ve_i] 
     \equiv \hat\ve_{i+1} 
    && \mod \langle \hat X, \hat \ve_{n-4-k},\ldots, \hat\ve_{i}\rangle 
    && \text{for}\ n-4-k\leq i \leq n-1+k
    \\
    &[\hat X, \hat \ve_i] 
     \equiv \hat\ve_{i+1} 
    && \mod \langle \hat X, \hat \ve_{n-4-k},\ldots, \hat\ve_{i},\hat \eta\rangle 
    && \text{for}\ n-k\leq i \leq 2n-7
    \\
    &[\hat \ve_i,\hat\ve_j]
     \equiv 0
    && \mod \langle \hat X, \hat \ve_{n-4-k},\ldots,\hat \ve_{i+j-n+4+k}\rangle
    &&  \text{for}\ n-4-k\leq i < j\ \text{with}\ i+j<2n-5
    \\
    &[\hat \ve_i,\ve_{2n-5-i}] 
     \equiv \hat \eta 
    && \mod \langle \hat X, \hat \ve_{n-4-k},\ldots,\hat \ve_{n-1+k}\rangle 
    &&  \text{for}\ n-4-k\leq i \leq n-3
    \\
    &[\hat \ve_i,\hat\ve_j] 
     \equiv 0 
    && \mod \langle \hat X, \hat \ve_{n-4-k},\ldots,\hat \ve_{i+j-n+4+k},\hat\eta\rangle
    &&  \text{for}\ n-4-k\leq i < j\leq 2n-6\ 
    \\
    & && && \quad \text{with}\ 2n-5 < i+j
    \\
    &[\hat X,\hat \eta] 
     \equiv 0 
    && \mod \langle \hat X,\widetilde\ve_{n-4-k},\ldots, \hat \ve_{n+k},\hat\eta\rangle
    \\
    &[\hat \ve_i,\hat \eta]
     \equiv 0 
    && \mod \langle \hat X,\hat\ve_{n-4-k},\ldots, \hat\ve_{i+4+2k},\hat\eta\rangle
    &&   \text{for}\ n-4-k\leq i \leq 2n-6
\end{align*}
which demonstrates that the Tanaka symbol of $\pr (E)$ is $\mathfrak{s}^{(k+1),n}$, as desired.

In the case of the Tanaka flat distribution $D(\mathfrak{s}^{k,n})$, one can choose the frame \eqref{skn frame} so that all the relations hold without modulus. The function $c$ is then uniformly zero, and the frame  \eqref{sk+1n frame} enjoys its relations without modulus. Hence, the distribution $\pr D(\mathfrak{s}^{k,n})$ is locally equivalent at a generic point to $D(\mathfrak{s}^{k+1,n})$.
\end{proof}

In order to investigate normalization conditions for $\mathfrak{s}^{k,n}$, we must first consider its universal Tanaka prolongation.

\begin{lemma}
\label{prolong_k_lemma}
    For $n>5$ and $0\leq k< n-4$, the universal Tanaka prolongation $\mathfrak{g}(\mathfrak{s}^{k,n})$ of $\mathfrak{s}^{k,n}$ is isomorphic to $\mathfrak{gl}_2(\R)\rtimes\mathfrak{heis}_{2n-5}$,
    \begin{equation}
      \label{sym_s^k}
\mathfrak{g}(\mathfrak{s}^{k,n})\cong\mathfrak{gl}_2(\R)\rtimes\mathfrak{heis}_{2n-5},  
    \end{equation}
     with grading \eqref{skn_symbol_grading}.
\end{lemma}
\begin{proof}
In \cite[Theorem 3]{doubrov2009local}, it is shown that for $n>5$, there is a unique maximally symmetric $(2,n)$ distributional germ with $5$-dimensional cube with infinitesimal symmetry algebra isomorphic to $\mathfrak{gl}_2(\R)\rtimes\mathfrak{heis}_{2n-5}$. 
The Tanaka symbol of this germ is isomorphic to $\mathfrak{s}^{0,n}$ and this germ is locally isomorphic to the flat distribution $D(\mathfrak{s}^{0,n})$ (which also follows from item (2) of Remark \ref{flat_symm}). Hence, by item (1) of  Remark \ref{flat_symm} it follows that 
\begin{equation}
\label{sym_s^0}
\mathfrak{g}(\mathfrak{s}^{0,n})\cong\mathfrak{gl}_2(\R)\rtimes\mathfrak{heis}_{2n-5}.
\end{equation}

Further, by the last  statement of Lemma  \eqref{prolongation_synbol_lemma}, for $0\leq k< n-4$, $\pr^kD(\mathfrak{s}^{0,n})\cong D(\mathfrak{s}^{k,n})$ locally at generic point of the former. Since the infinitesimal symmetry algebras of $D(\mathfrak{s}^{0,n})$ and $\pr^kD(\mathfrak{s}^{0,n})$ at generic points of the latter  are isomorphic, item (1) of Remark  \ref{flat_symm} together with \eqref{sym_s^0} implies \eqref{sym_s^k}. The grading of $\mathfrak{gl}_2(\R)\rtimes\mathfrak{heis}_{2n-5}$ must be as in \eqref{skn_symbol_grading} because of the uniqueness of the largest algebra satisfying items (1) and (2) of Definition \ref{universal_prol_def}. 
\end{proof}


\begin{theorem}
\label{Nonexistence Theorem}
        For $n>5$ and $k\leq n-5$, the Tanaka symbol $\mathfrak{s}^{k,n}$ does not admit a linear invariant normalization condition.
\end{theorem}
\begin{proof}
Let $\mathcal{B}$ be the basis $(Y,H,E,X,\ve_1,\ldots,\ve_{2n-5},\eta)$ for $\mathfrak{g}(\mathfrak{s}^{k,n})$ and let $L_H, L_E: \mathcal{B}\to \Z$ be the functions which take an element of $\mathcal{B}$ to its $H$- and $E$-eigenvalue, respectively. The graded weights and eigenvalues of all elements of $\mathcal{B}$ are given in the following table:
    \begin{equation}
    \label{weights}
        \begin{tabular}{||c | c | c | c | c | c | c ||} 
            \hline
            & $\ve_i$ & $\eta$ & $X$ & $H$ & $E$ & $Y$ \\ [0.5ex] 
            \hline\hline
            $\text{wght}$&$n-4-k-i$&$-3-2k$&$-1$&$0$&$0$&$1$ \\ [1ex] 
            \hline
            $L_H$&$2i+5-2n$&$0$&$2$&$0$&$0$&$-2$ \\ [1ex] 
            \hline
            $L_E$&1&2&0&0&0&0\\ [1ex] 
            \hline
        \end{tabular}
\end{equation}
Write
    \begin{gather}
       \mathcal{B}^1 = \big\{A^*\otimes B: A,B \in \mathcal{B},\ \text{wght}(A)<0\big\}\quad\text{and}
       \\
       \mathcal{B}^2=\big\{A^*\wedge B^*\otimes C: A,B,C\in \mathcal{B},\ \text{wght}(A)<0,\ \text{wght}(B)<0\big\},
    \end{gather}
    the induced bases for $C^1_+\big(\mathfrak{s}^{k,n},\mathfrak{g}(\mathfrak{s}^{k,n})\big)$ and $C^2_+\big(\mathfrak{s}^{k,n},\mathfrak{g}(\mathfrak{s}^{k,n})\big)$. Also, for $\lambda\in \Z$, write
    \begin{gather}
        C_+^k\big(\mathfrak{s}^{k,n},\mathfrak{g}(\mathfrak{s}^{k,n})\big)_\lambda = \big\{c\in C^k_+\big(\mathfrak{s}^{k,n},\mathfrak{g}(\mathfrak{s}^{k,n}\big): \text{ad}(H)(c) = \lambda c\big\}.
    \end{gather}
    Since the coboundary operator $\partial$ is invariant under $H$, we have that $\partial C_+^1\big(\mathfrak{s}^{k,n},\mathfrak{g}(\mathfrak{s}^{k,n})\big)_\lambda \subseteq C_+^2\big(\mathfrak{s}^{k,n},\mathfrak{g}(\mathfrak{s}^{k,n})\big)_\lambda$. For any $n\geq 6$,
    \begin{equation}
    \label{small wght space}
        C_+^1\big(\mathfrak{s}^{k,n},\mathfrak{g}(\mathfrak{s}^{k,n})\big)_{2(n+k-3)} = \big\langle \ve_{n-3-k}^*\otimes \ve_{2n-6}\big\rangle.
    \end{equation}
    To see this, observe that $L_H(\ve_{2n-6}) = 2n-7$ is the maximal value of $L_H$ on $\mathcal{B}$ and $L_H(\ve_{n-3-k})=-2k-1$ is the minimal value of $L_H$ when restricted to the negatively graded elements of $\mathcal{B}$.
    
    Also note that 
    \begin{equation}
    \label{close_eigenspace}
    \partial(\ve_{n-3-k}^*\otimes \ve_{2n-6}) = 0.
    \end{equation}
    This easily follows from the definition of the coboundary map and the following two properties of the algebra  $\mathfrak{g}(\mathfrak{s}^{k,n})$: 
\begin{enumerate}
    \item $\ve_{n-3-k}$ cannot be represented as a commutator of two negatively graded elements;
    \item $\ve_{2n-6}$ commutes with a negatively graded  part of $\mathfrak{g}(\mathfrak{s}^{k,n})$ in the case under consideration, where $k\geq n-5$\footnote {Relation \eqref{close_eigenspace} is also true for $k=n-4$, despite the fact that $\ve_1$ is has negative graded weight and does not commute with $\ve_{2n-6}$.}
    \end{enumerate}

    Now suppose for contradiction that $\mathcal{N}$ is an invariant normalization condition for the prolonged Tanaka symbol $\mathfrak{g}(\mathfrak{s}^{k,n})$; write $\partial$ for the coboundary map $C_+^1\big(\mathfrak{s}^{k,n},\mathfrak{g}(\mathfrak{s}^{k,n})\big)\to C_+^2\big(\mathfrak{s}^{k,n},\mathfrak{g}(\mathfrak{s}^{k,n})\big)$. Since $\mathcal{N}$ and $\im\partial$ are invariant under the diagonalizable map $\text{ad}(H)$, they decompose into eigenspaces $\mathcal{N}_\lambda$ and $\im(\partial)_\lambda$ for the action of $H$:
    \[
        \mathcal{N} = \bigoplus_\lambda \mathcal{N}_\lambda\quad\text{and}\quad \im(\partial) = \bigoplus_\lambda \im(\partial)_\lambda.
    \]
    so that for each $\lambda \in \mathbb{Z}$,
    \[
        C_+^2\big(\mathfrak{s}^{k,n},\mathfrak{g}(\mathfrak{s}^{k,n})\big)_\lambda = \mathcal{N}_\lambda \oplus \im(\partial)_\lambda
    \]
    Since $\text {ad}\,H$ commutes with $\partial$, we have 
    \[
        \im(\partial)_{\lambda}=\partial\Bigl(C_+^1\big(\mathfrak{s}^{k,n},\mathfrak{g}(\mathfrak{s}^{k,n})\big)_\lambda\Bigr)
    \]
    In particular, for $\lambda=2(n+k-3)$ from  \eqref{small wght space} it follows that $\im(\partial)_{\lambda}=\big\langle \partial (\ve_{n-3-k}^*\otimes \ve_{2n-6})\big\rangle = 0$, which implies $\mathcal{N}_{\lambda}= C_+^2\big(\mathfrak{s}^{k,n},\mathfrak{g}(\mathfrak{s}^{k,n})\big)_{\lambda}$. This eigenspace includes the cochain $\ve_{n-3-k}^*\wedge\eta^*\otimes \ve_{2n-6}$, which is thus included in $\mathcal{N}$. However, computing the action of $\ve_1$ and $\ve_{n-4-k}$ on $(\mathfrak{s}^{k,n}_-)^*$, we have on a basis
    \begin{gather*}
        \ad(\ve_{1}) =
        \begin{cases}
            X^*\mapsto 0
            \\
            \ve_i^*\mapsto 0\quad \text{for}\ n-3-k\leq i \leq 2n-6
            \\
            \eta^*\mapsto \ve_{2n-7}^*
        \end{cases}
        \\
        \ad(\ve_{n-4-k}) =
        \begin{cases}
            X^*\mapsto -\ve_{n-3-k}^*
            \\
            \ve_i^*\mapsto -\delta_i^{n-3-k}X^*\quad \text{for}\ n-3-k\leq i \leq 2n-6
            \\
            \eta^*\mapsto (-1)^{n+1-k}\ve_{n-1+k}^*
        \end{cases}
    \end{gather*}
    Using the Leibniz rule, we then have
    \begin{gather*}
        \ad(\ve_1)\Big(\ad(\ve_{n-4-k})^2(\ve_{n-3-k}^*\wedge\eta^*\otimes  \ve_{2n-6})\Big)
        = \ad(\ve_1)\Big(2(-1)^{n-k}X^*\wedge\ve_{n-1+k}^*\otimes \ve_{2n-6}\Big)
        \\
        = 2(-1)^{n-k+1}X^*\wedge\ve_{n-1+k}^*\otimes \eta = -2\cdot\partial\big(X^*\otimes \ve_{n-4+k}\big)
    \end{gather*}
    Since $\mathcal{N}$ is assumed to be invariant under $\mathfrak{g}(\mathfrak{s}^{k,n})$, the above element is also included in $\mathcal{N}$. This implies that $\mathcal{N}\cap \im(\partial)$ is nonzero, contradicting the definition of $\mathcal{N}$.
\end{proof}

Together, Theorems \ref{Tanaka Symbol}, \ref{Existence Theorem}, \ref{Unification Theorem}, and \ref{Nonexistence Theorem} yield the following.

\begin{prop}
    Let $D$ be a $(2,n)$ distribution with $5$-dimensional cube. Then at a generic point, the $(n-5)$th iterated Cartan prolongation $\pr^{(n-5)}(D)$ has constant Tanaka symbol $\mathfrak{s}^{(n-5),n}$, but this symbol does not admit a linear invariant normalization condition.
 
    On the other hand, at a generic point the $(n-4)$th Cartan prolongation $\pr^{(n-4)}(D)$ has Tanaka symbol $\mathfrak{s}^{(n-4),n}$. This symbol admits a linear invariant normalization condition.
\end{prop}
    Therefore $\Symp(D)$ is locally equivalent to $\pr^{n-4}(D)$, which is the first Cartan prolongation where Theorem \ref{Cartan Connection} can be used to construct a Cartan connection for each $(2,n)$ distribution $D$ with $5$-dimensional cube.
    
\bibliographystyle{plain}
\bibliography{Bibliography.bib}
\end{document}